\documentclass[10pt,a4paper]{article}
\usepackage[utf8]{inputenc}
\usepackage{amsmath}
\usepackage{amsfonts}
\usepackage{amssymb}
\usepackage{amsthm}
\usepackage{enumitem}
\usepackage{color}
\usepackage{hyperref}
\usepackage{graphicx}

\usepackage[a4paper, textwidth=13cm]{geometry}

\providecommand{\R}{\mathbb{R}}
\providecommand{\N}{\mathbb{N}}
\providecommand{\Z}{\mathbb{Z}}

\providecommand{\rats}{\overset{t.s.}{\longrightarrow}}

\providecommand{\eps}{\epsilon}

\newcommand{\spaceHn}{\mathcal{H}_{n}}
\newcommand{\spaceHndiv}{\mathcal{H}_{n,\mathrm{div}}}
\newcommand{\spaceV}{\mathcal{V}}
\newcommand{\x}{\bar{x}}
\newcommand{\ueps}{u^{\epsilon}}
\newcommand{\peps}{p^{\epsilon}}
\newcommand{\oem}{\Omega_{\epsilon}^M}
\newcommand{\weps}{w_{\epsilon}}
\newcommand{\veps}{v^{\epsilon}} 
\newcommand{\bxfxe}{\left(\bar{x},\dfrac{x}{\epsilon}\right)}
\newcommand{\foe}{\dfrac{1}{\epsilon}}
\newcommand{\tweps}{\tilde{w}_{\epsilon}}

\DeclareMathOperator{\diver}{div}

\DeclareMathOperator{\Span}{Span}

\let\div\diver

\newtheorem{thm}{Theorem}[section]

\newtheorem{lem}[thm]{Lemma}
\newtheorem{prop}[thm]{Proposition}

\newtheorem{rem}[thm]{Remark}

\newtheorem{dfn}[thm]{Definition}

\newtheorem{example}[thm]{Example}

\title{Homogenization and dimension reduction of the Stokes-problem with Navier-Slip condition in thin perforated layers}
\date{}
\author{J. Fabricius$^{*}$ and M. Gahn$^{**}$}

\begin{document}

\maketitle

\begin{abstract}
 We study a Stokes system posed in a thin perforated layer with a Navier-slip condition on the internal oscillating boundary from two viewpoints: 1) dimensional reduction of the layer and 2) homogenization of the perforated structure. Assuming the perforations are periodic, both aspects can be described through a small parameter $\eps>0,$ which is related to the thickness of the layer as well as the size of the periodic structure. By letting $\eps$ tend to zero, we prove that the sequence of solutions converges to a limit which satisfies a well-defined macroscopic problem. More precisely, the limit velocity and limit pressure satisfy a two pressure Stokes model, from which a Darcy law for thin layers can be derived. Due to non-standard boundary conditions, some additional terms appear in Darcy's law.
\end{abstract}
\vspace{3em}

\\


\let\thefootnote\relax\footnotetext{$^{*}$Department of Engineering Sciences and Mathematics, Lule{\aa} University of Technology, 97187 Lule{\aa}, Sweden, john.fabricius@ltu.se.}

	\vspace{.5mm}

\let\thefootnote\relax\footnotetext{$^{**}$Interdisciplinary Center for Scientific Computing, University of Heidelberg, Im Neuenheimer Feld
	205, 69120 Heidelberg, Germany, markus.gahn@iwr.uni-heidelberg.de.}

\section{Introduction}\label{sec:intro}

The present work is concerned with the mathematical derivation of Darcy's law for incompressible newtonian flow in thin periodically perforated layers. On the inner boundary of the layer we assume a Navier-slip condition, and on the external boundary a normal stress boundary condition. Let $\Omega^\eps$ denote a domain in $\mathbb{R}^n$ ($n\geq 2$), to be specified below. The fluid velocity $\ueps : \Omega^{\epsilon} \rightarrow \R^n$ and the fluid pressure $\peps:\Omega^{\eps} \rightarrow \R$ are assumed to be governed by the Stokes system:
\begin{equation}\label{Stokes:eq}
\left\{
\begin{aligned}
-\nabla\cdot\bigl(-p^\eps I+2\mu D(u^\eps)\bigr)& =f^\eps && \text{in }\Omega^\eps\\
\nabla\cdot u^\eps& =0 && \text{in }\Omega^\eps,\\
\end{aligned}
\right.
\end{equation}
where the constant $\mu>0$ is the fluid viscosity,
\begin{align*}
    D(u^\eps):=\frac{1}{2}\left(\nabla u^\eps+(\nabla u^\eps)^T\right)
\end{align*}
is the symmetric part of the velocity gradient (or strain-rate tensor) and $f^\eps$ is a given function representing force per unit volume.

The fluid domain $\Omega^\eps$ is a periodically perforated layer with thickness of order $\epsilon$. It can be described by  scaling and periodic repetition of a reference cell $Y$ which is partitioned into a solid phase $Y_s$ and a fluid phase $Y_f$ (a precise definition is given in Section~\ref{Micro:sec}). Further, we denote the intersection of the fluid part $\overline{Y_f}$ with the upper resp. lower boundary of $Y$ by $S_N^+$ resp. $S_N^-$.
We divide the boundary $\partial \Omega^{\epsilon}$ into two parts, an external part $\Gamma_N^{\epsilon}$ and an internal boundary $\Gamma_D^{\epsilon}$ (the interface with the solid perforations in the thin layer). 
For the Stokes system \eqref{Stokes:eq} we impose on the interior boundary of the perforated layer a Navier-slip condition,
\begin{equation}\label{slip:bc}
\left\{
\begin{aligned}
u^\eps \cdot\hat{n}& =0&& \text{on }\Gamma_D^\eps\\
2\mu [D(u^\eps) \hat{n} ]_{\tau} + \alpha \epsilon^{\gamma} u^\eps &= g^\eps && \text{on } \Gamma_D^\eps
\end{aligned},
\right.
\end{equation}
where $g_\eps$ is a given function that represents a shear stress (force per unit area) and $\alpha,$ $\gamma$ are given constants. As usual, $\hat{n}$ denotes the outward unit normal and $z_\tau $ denotes the tangential component of a vector $z \in \R^n$ relative to $\hat{n}$, i.e., $z_\tau = z - (z\cdot \hat{n})\hat{n}$. Further, we have $\alpha \geq 0$ and $\gamma \in \R$.

From a physical point of view the reciprocal of $\alpha \epsilon^{\gamma}$ is the slip length, and the parameter $\gamma$ is a measure for the slip of the the fluid at the inner boundary. For example, the Navier-slip condition describes an effective interface condition between a free fluid and a porous medium, known as the Beavers-Joseph condition, which was derived experimentally in \cite{BeaversJoseph1967}, and rigorously from first principles for a flat interface in \cite{mikelic2000interface}. The parameters $\alpha$ and $\gamma$ are related to the porosity of the porous medium, and especially fit to the cases $\gamma \geq -1$, see \cite[Proposition 8]{mikelic2000interface}. The case $\gamma<-1$ corresponds to the slip length being negligible. More precisely, when $\gamma<-1,$ we shall see that condition \eqref{slip:bc} becomes a ``true'' no-slip in the limit as $\eps$ tends to zero, i.e. all components of the velocity vanishes. For an in-depth discussion of slip vs. no-slip see e.g. \cite{Bazant2008} and \cite[Ch~6.4]{Panton2013}.

On the exterior boundary we prescribe the stress vector (traction) as a normal stress, i.e.
\begin{equation}\label{traction:bc}
\bigl(-p^\eps I+2\mu D(u^\eps)\bigr)\hat{n} =-p_{b}^\eps\,\hat{n}\quad\text{on }\Gamma_N^\eps,
\end{equation}
where the function $p_b^\eps$ represents an external pressure, which is caused for example from the mechanical contact of the thin layer with a surrounding fluid. By letting $\epsilon$ tend to zero we study the convergence of the solutions $(u^\eps,p^\eps)$ and show that the limit can be characterized by a triple $(u^0,p^0,p^1)$ which is the unique solution of a two-scale homogenized problem, also called the two pressure Stokes model.  Both the dimensional reduction and the homogenization is achieved by the well-known two-scale convergence method. The limit pressure $p^0$ satisfies the $(n-1)$-dimensional Darcy law including some additional terms whose origins can be traced to the boundary conditions. More precisely, the models differ with regard to the following cases:
\begin{itemize}
    \item $\alpha >0$ and $\alpha = 0$,
    \item $\gamma <-1$, $\gamma = -1$, and $\gamma >-1$,
    \item $S_N^{\pm} = \emptyset $ and $S_N^{\pm} \neq \emptyset$.
\end{itemize}

The first rigorous derivation of Darcy's law for periodically perforated domains was given by Tartar in \cite{SanchezPalencia1980} who considered the system \eqref{Stokes:eq}, where $\Omega^\eps$ is defined by periodic repetition in all coordinate axis directions and scaling of the cell $Y=(0,1)^n,$ together with the boundary condition $u^\eps=0$ on $\partial\Omega^\eps.$ To prove a $L^2$-bound for the pressure Tartar constructed a so called \emph{restriction operator} for velocity fields. This locally defined operator combined with a duality argument, gives a local extension of the pressure. The uniform bound is inferred from the extended pressure. An explicit formula for the pressure extension was later obtained by Lipton and Avellaneda \cite{lipton1990darcy}.
For another proof for the derivation of Darcy's law not using the restriction operator, and also some more discussions about related literature, we refer to \cite{chechkin2007homogenization}, where a uniform bound for the pressure is given directly using a Bogovski\u{\i} operator in perforated domains.
To construct the restriction operator, Tartar assumed that $Y_s$ is compactly contained in $Y$, i.e. it must be possible to choose the reference cell so that $\mathrm{dist}\,(\partial Y,\partial Y_s)>0$. In particular, the construction fails when the solid part is connected. However, by modifying Tartar's proof Allaire \cite{allaire1989homogenization} was able to define a restriction operator for a periodic porous medium with connected solid phase. To retain all the properties of Tartar's operator, Allaire involved a global projection operator in the construction. In this work, we simplify Allaire's construction by showing that some properties of the restriction operator can be relaxed without compromising its objective. More precisely, we construct a local restriction operator for thin layers that provides a bound for the pressure $p^\eps$ in $L^2(\Omega_\epsilon)$ which is independent of the parameter $\epsilon.$

Another novelty of the present study consists in the choice of boundary conditions.
Models with similar boundary conditions in a periodically perforated domain with strict inclusions of size $a_{\epsilon}$ much smaller than $\epsilon$ were considered by Allaire in \cite{allaire1991homogenization} in the case $\alpha > 0$. A similar model with holes of size $\epsilon$ was studied by Cioranescu et~al.  in \cite{cioranescu1996homogenization}. However, as mentioned in \cite{capatina2011homogenisation}, there was a mistake in the choice of oscillating test-functions, which fulfill a no-slip condition. In \cite{capatina2011homogenisation} they addressed this problem again, but it seems that there is also an error in the case $\gamma < -1$, where they state that the Darcy-velocity becomes zero. More precisely, it is claimed that the unfolded sequence of the traces of the microscopic fluid velocities converge to the Darcy velocity, which is in general not true. In fact, we show that for $\gamma < -1$ the two-scale limit of the microscopic velocity fulfills a no-slip boundary condition instead of the zero normal flux boundary condition.
 In contrast to our paper and \cite{allaire1991homogenization}, where a stress of the form $\peps I - 2\mu D(\ueps)$ including the symmetric gradient is used, in \cite{cioranescu1996homogenization,capatina2011homogenisation} they use the stress $\peps I - \nabla \ueps$. Hence, in our weak formulation we obtain integrals including the symmetric gradient, which makes the use of a Korn inequality necessary to establish coercivity of the microscopic problem. The essential difference to the previously mentioned works is that we consider a thin perforated layer, leading to simultaneous dimension reduction and homogenization, as well as the critical case $\alpha = 0$.  Additionally, we consider the normal traction condition \eqref{traction:bc}. The corresponding homogenization result for the Stokes system in a periodically perforated domain with a no slip condition on the internal oscillating boundary can be found in \cite{fabricius2017homogenization}. 
We emphasize, that in our paper we especially consider connected solid phases, leading to additional technical difficulties.

For the derivation of the macroscopic model we make use of the two-scale convergence in thin domains, which was first introduced by \cite{MarusicMarusicPalokaTwoScaleConvergenceThinDomains} for homogeneous layers and later extended to perforated layers in \cite{NeussJaeger_EffectiveTransmission}. Two-scale compactness results are based on $\epsilon$-uniform \textit{a priori} estimates for the microscopic solution, which also guarantee existence and uniqueness of a weak solution for the micro-model. Since the associated variational equation includes symmetric gradients, a Korn inequality with explicit dependence on the scaling parameter $\epsilon$ is necessary. In \cite[Lemma 2.4]{allaire1991homogenization} a Korn inequality with an additional boundary term was shown, without using the normal zero trace. However, this result is only enough to deal with the case $\alpha > 0$ and $\gamma$ small enough. To overcome this problem we prove an $\epsilon$-uniform Korn inequality for functions with normal trace equal to zero, where we need an additional assumption on the geometry of the solid phase $Y_s$ to  avoid cylindrical inclusions touching the upper and lower boundary of $Y$ or its lateral boundary. An additional challenge is to obtain a uniform $L^2$-estimate for the pressure. For this we construct a restriction operator for perforated thin layers including a connected solid phase, where we were able to simplify some technical parts from the proof of \cite{allaire1989homogenization}.
With the \textit{a priori} estimates for the fluid velocity and pressure, we are able to pass to the limit $\epsilon \to 0$ in the microscopic problem by applying general two-scale compactness result. It is well known in the homogenization theory for the Stokes problem, that in the two-scale limit a so called two pressure homogenized model with weak solution $(u^0,p^0,p^1)$ is obtained. To identify the two pressures $p_0$ and $p_1$, an orthogonality argument can be used, see \cite[Lemma 1.5]{allaire1997one} or \cite[Lemma 14.3]{chechkin2007homogenization} for the case of a no slip boundary condition. We generalize this result to the case of thin layers and normal zero boundary conditions. More precisely,  denoting by $L^2(\Sigma,\spaceHn)$ the space of functions in  $L^2(\Sigma, H^1(Y_f)^n)$ with vanishing normal trace and $(0,1)^{n-1}$-periodic,  we show that every functional on $L^2(\Sigma,\spaceHn)$ which vanishes on the subspace of solenoidal vector fields with respect to the micro- and macro-variable (the latter after averaging over the reference element $Y_f$), can be decomposed into the sum of the gradients of the two pressures $\nabla_{\x} p^0$ and $\nabla_y p^1$. The two pressure Stokes model includes all necessary information for the macroscopic and microscopic scale to obtain the Darcy-law for the pressure $p^0$. In summary, we have the following novel contribution in our paper:
\begin{itemize}
    \item Korn inequalities for vector fields with vanishing normal zero trace and explicit dependence on scaling parameter $\epsilon$.
    \item Restriction operator for thin perforated layers with a connected solid phase.
    \item Two pressure decomposition for functionals on the space $L^2(\Sigma,\spaceHn)$.
    \item Derivation of effective models and associated Darcy-laws for the cases $\alpha = 0$ and $\alpha > 0$, and the different choices of $\gamma \in \R$.
\end{itemize}

The paper is organized as follows. In Section~\ref{Micro:sec} we give a detailed description of the microscopic model including assumptions on the given data and model parameters. The macroscopic model and the main convergence result is stated in Section~\ref{Main:sec}. Existence and uniqueness of solutions to the micro problem as well as \textit{a priori} estimates are proven in Section~\ref{SectionExistenceApriori}. To obtain a uniform bound for the pressure, we construct a restriction operator similar to the operator defined in \cite{allaire1989homogenization}, but with weaker properties and a different proof. The construction relies heavily on the so called Bogovski\u{\i} operator recalled in Appendix~\ref{SectionAppendixBogovskii}. The splitting of the pressure in two parts $p^0$ and $p^1$ in the macro model can be seen as a decomposition of functionals, which is proven in Section~\ref{SectionTwoPressureDecomposition}. The macro model is derived in Section~\ref{Derivation:sec} based on two-scale convergence for thin layers. 
This notion of convergence is defined in Appendix~\ref{SectionTSConvergence} and some basic compactness results are given.

\section{The microscopic model}\label{Micro:sec}

In this section we elaborate further on the microscopic model and its underlying geometry. As mentioned above, we consider a thin perforated layer consisting of a connected fluid part and a fixed solid part, which is in general also connected. The fluid flow is governed by the Stokes system, where on the fluid-solid interface we consider a Navier-slip boundary condition.

\subsection{The microscopic geometry and notations}\label{sec:not}

We give now the precise definition of the domain $\Omega^\eps.$ Let $E_s$ denote an unbounded non-empty open subset of the layer $\R^{n-1}\times (-1,1).$ The set $E_s,$ called the solid phase, is assumed to be invariant under integer translations parallel to the first $(n-1)$ coordinate axes, i.e.
\[
x\in E_s\implies x+k\in E_s
\]
for all $k\in \Z^{n-1}\times \{0\}.$ The fluid phase of the layer is denoted as
\[
E_f:=\R^{n-1}\times (-1,1)-\overline{E_s},
\]
which is assumed to be a (unbounded) Lipschitz domain in $\R^n.$ Consequently, the reference layer cell
\begin{equation*}
Y:=(0,1)^{n-1} \times (-1,1),
\end{equation*}
can be partitioned into a solid phase $Y_s:=Y\cap E_s$ and a fluid phase $Y_f:=Y\cap E_f.$ We assume that both $Y_f$ and $Y_s$ are open. The boundary with zero normal flux condition of $Y_f$ is defined as
\begin{equation*}
\Gamma_D=Y\cap \partial Y_s
\end{equation*}
so that
\[
Y=Y_f\cup Y_s\cup \Gamma_D,\quad Y_f\cap Y_s=\emptyset,
\]
see Figure~\ref{cell:fig} for an example. 
In addition we assume that $Y_f$ is a Lipschitz domain.
Further, we define the upper and lower boundary of the cell $Y$ by 
\begin{align*}
    S^{\pm} := (0,1)^{n-1} \times \{\pm 1\}.
\end{align*}
Now, we denote the upper and lower part of the boundary of $Y_f$ by
\begin{align*}
    S^{\pm}_N:= \mathrm{int} \left( S^{\pm} \cap \partial Y_f \right).
\end{align*}
\begin{figure}[ht]\label{cell:fig}
\centering
\includegraphics[height=5cm]{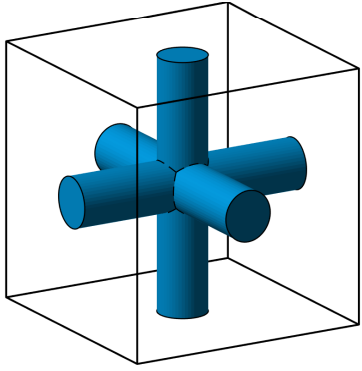}
\caption{Reference layer cell $Y$ in $\R^3$ with solid phase $Y_s$ (colored area)}
\end{figure}

\begin{figure}[ht]\label{layer:fig}
\centering
\includegraphics[height=5cm]{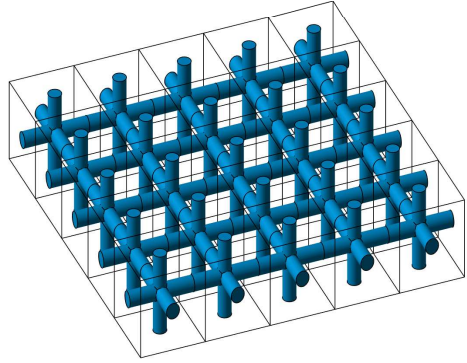}
\caption{Thin layer $\Omega^\eps_M$ in $\R^3$ with solid phase (colored area) connected}
\end{figure}
We will see that we have to distinguish the two cases $S^{\pm}_N$ empty or not, and obtain different homogenized problems or more precisely different kind of effective coefficients. We emphasize that both cases are important for applications. Considering only a single layer would lead to a fluid which is enclosed by the solid ($S^{\pm}_N$ is empty). In many applications, for example transport problems through membranes, the layer is coupled to regions above or below, leading to the case $S^{\pm}_N \neq \emptyset$. However, in this case the situation gets more complicated, since one has to take into account additional coupling conditions to other regions.

For given constants $L_1,L_2,\dotsc,L_{n-1} \in \N,$ let $\Sigma$ denote the fixed domain in $\mathbb{R}^{n-1}$ defined by
\begin{equation*}
\Sigma:=(0,L_1) \times \dotsb \times (0,L_{n-1}).
\end{equation*}
Choosing $0<\eps\ll 1$ so that $\epsilon^{-1} \in \N,$ we obtain a thin layer 
\begin{equation*}
\Omega^\eps_M:=\Sigma\times (-\eps,\eps).
\end{equation*}
The domain $\Omega^\eps$ is then defined as the fluid phase of $\Omega_M^\eps,$ i.e. 
\begin{equation*}
\Omega^\eps:=\Omega_M^\eps\cap \eps E_f.
\end{equation*}
Alternatively, $\Omega^\eps$ could be defined by periodic repetition and uniform scaling of the set $Y_f,$ i.e.
\[
\Omega^\eps:=\mathrm{int}\,\left(\bigcup_{k} \overline{\eps(Y_f+k)}\right),
\]
where the union is taken over all $k\in \mathbb{Z}^{n-1}\times \{0\}$ such that $\eps(Y+k)\subset \Sigma_M^\eps.$ The boundary of $\Omega^\eps$ is divided into two disjoint sets $\Gamma_D^\eps$ and $\Gamma_N^\eps,$ where 
\begin{align*}
\Gamma_N^\eps& :=\mathrm{int}(\partial\Omega^\eps\cap \partial\Omega_M^\eps) && \text{(exterior boundary)}\\
\Gamma_D^\eps& :=\partial\Omega^\eps \cap \Omega_M^\eps && \text{(interior boundary)}.
\end{align*}
Note that the exterior boundary $\Gamma_N^\eps$ may have holes of solid phase.
Obviously, we have
\[
\lvert \Omega^\eps\rvert=\mathcal{O}(\eps),\quad \lvert \Gamma_D^\eps\rvert= \mathcal{O}(1),\quad \lvert \Gamma_N^\eps\rvert= \mathcal{O}(1)\quad \text{as $\eps\to 0.$}
\]
Obviously, we assume that $\Omega_\epsilon$ is connected. It is important to note that the solid phase $E_s$ may also be connected, which allows physically realistic structures to be considered, as shown in Figure~\ref{layer:fig}.

Finally, let us introduce some basic function spaces and notations. For an open set $U\subset \R^m$ with $m \in \N$ we denote by $L^2(U)$ the usual Lebesgue space, and by $H^1(U)$ the Sobolev space of functions in $L^2(U)$ with weak derivatives in $L^2(U)$.
We also introduce the space of periodic functions with respect to the first $(n-1)$ components (the lateral boundary of the reference element $Y$):
\begin{align}\label{SpaceZeroTracesPartBoundary}
    H_{\#}^1(Y_f) := \left\{ u \in H^1(Y_f) \, : \, u \mbox{ is } (0,1)^{n-1}\mbox{-periodic}\right\}.
\end{align}
If $U$ is a bounded Lipschitz domain and $\Gamma \subset \partial U$, then we define the space of Sobolev functions with zero traces on $\Gamma$ by
\begin{align*}
    H^1(U,\Gamma):= \left\{ u \in H^1(U) \, : \, u = 0 \mbox{ on } \Gamma \right\}.
\end{align*}
For $x \in \Omega_M^{\epsilon}$ we denote the first $(n-1)$ components with  $\x:=(x_1,\ldots,x_{n-1})^T$.
For a function $p \in H^1(\Sigma)$ we denote its gradient by $\nabla_{\x}  p \in L^2(\Sigma )^{n-1}$ and also consider this function as an element of $L^2(\Sigma)^n$ via the natural embedding $ (\partial_1 p ,\ldots, \partial_{n-1} p,0)^T$. We use the same notation $\nabla_{\x} p$ for the element in $L^2(\Sigma)^{n-1}$ and $L^2(\Sigma)^n$ and it should be clear from the context in which sense it has to be understood.  Further, for a vector field $v \in L^2(\Sigma)^n$ we define the weak divergence with respect to $\x$ as the function $\nabla_{\x} \cdot v \in L^2(\Sigma)$ (if it exists) such that for all $\phi \in C_0^{\infty}(\Sigma)$ it holds that
\begin{align*}
    \int_{\Sigma} \nabla_{\x} \cdot v \phi d\x = - \int_{\Sigma} v \cdot \nabla_{\x} \phi d\x.
\end{align*}
Obviously, the $n$-th component of $v$ has no influence on the divergence with respect to $\x$.

\subsection{The microscopic model}

We pose the following boundary value problem for the velocity field $u^\eps$ and the pressure field $p^\eps$ in $\Omega^\eps.$
\begin{align}
\label{MicroModelNavierSlip}
\left\{
\begin{aligned}
-\nabla\cdot \left(-p^\eps I+2\mu D(u^\eps)\right)& =f^\eps && \text{in }\Omega^\eps\\
\nabla\cdot u^\eps& =0 && \text{in }\Omega^\eps\\
u^\eps \cdot\hat{n}& =0&& \text{on }\Gamma_D^\eps\\
2\mu [D(u^\eps) \hat{n} ]_{\tau} + \alpha \epsilon^{\gamma} u^\eps &= g^\eps && \text{on } \Gamma_D^\eps \\
\left(p^\eps I-2\mu D(u^\eps)\right)\hat{n}& =p_{b}^\eps\,\hat{n} && \text{on }\Gamma_N^\eps,
\end{aligned}
\right.
\end{align}
where $\alpha\geq 0$ and $\gamma \in \R$. Appropriate assumptions on the given data $f^\eps,$ $g^\eps$ and $p_b^\eps$ are stated below.

\begin{rem}
In the critical case $\alpha = 0 $ the structure of the problem $\eqref{MicroModelNavierSlip}$ changes significantly. For positive $\alpha$ the coercivity of the associated bilinear form can be  obtained from the boundary term including $\alpha \epsilon^{\gamma} \ueps$, see \cite{allaire1991homogenization} for more details. For the case $\alpha = 0$ we need a Korn inequality for fields with zero normal trace, see Lemma \ref{LemmaKornInequality} below.
\end{rem}

We have the following weak formulation for the micro model $\eqref{MicroModelNavierSlip}$:
Find $(u^\eps,p^\eps) \in H^1(\Omega^\eps)^n \times L^2(\Omega^\eps)$ with $\nabla \cdot u^\eps = 0$ and $u^\eps \cdot \hat{n} = 0$ on $\Gamma_D^\eps$, such that for all $\phi^\eps \in H^1(\Omega^\eps)^n$ with $\phi^\eps \cdot \hat{n} = 0$ on $\Gamma_D^\eps$ it holds that
\begin{align*}
    2\mu \int_{\Omega^\eps} & D(u^\eps) : D(\phi^\eps) dx + \alpha \epsilon^\gamma \int_{\Gamma_D^\eps} u^\eps \cdot \phi^\eps d\sigma - \int_{\Omega^\eps} p^\eps \nabla \cdot \phi^\eps dx 
    \\
    &= \int_{\Omega^\eps} f^\eps \cdot \phi^\eps dx + \int_{\Gamma_D^\eps} g^\eps \cdot \phi^\eps d\sigma - \int_{\Gamma_N^\eps} p_b^\eps \hat{n} \cdot \phi^\eps d\sigma.
\end{align*}
Under the assumption that $p_b^\eps$ is the trace of a function defined in $\Omega^\eps$, we can write
\begin{align}
    \begin{aligned}
    \label{MicroModelNavierSlipVarEqu}
    2\mu \int_{\Omega^\eps} & D(u^\eps) : D(\phi^\eps) dx + \alpha \epsilon^\gamma \int_{\Gamma_D^\eps} u^\eps \cdot \phi^\eps d\sigma - \int_{\Omega^\eps} (p^\eps - p_b^\eps) \nabla \cdot \phi^\eps dx 
    \\
    &= \int_{\Omega^\eps} (f^\eps - \nabla p_b^\eps) \cdot \phi^\eps dx + \int_{\Gamma_D^\eps} g^\eps \cdot \phi^\eps d\sigma .
    \end{aligned}
\end{align}

\noindent\textbf{Assumptions on the data:}

\begin{enumerate}[label = (A\arabic*)]
\item $f^\eps \in L^2(\Omega^\eps)^n$ with $f^\eps \rats f^0$ for $f^0 \in L^2(\Sigma)^n$, see Section \ref{SectionTSConvergence} in the appendix for the definition of the two-scale convergence. Especially, it holds that
\begin{align*}
    \frac{1}{\sqrt{\epsilon}} \Vert f^\eps\Vert_{L^2(\Omega^\eps)} \le C.
\end{align*}

\item\label{AssumptionPressure} It holds that $p_b^\eps \in H^1(\Omega^\eps)^n$ with 
\begin{align*}
    \frac{1}{\sqrt{\eps}} \Vert p_b^\eps \Vert_{H^1(\Omega^\eps)} \le C.
\end{align*}
Further, there exists $p_b \in H^1(\Sigma)$ and $p_b^1 \in L^2(\Sigma,H_{\#}^1(Y_f))$, such that
\begin{align*}
    p_b^{\eps} &\rats p_b,
    \\
    \nabla p_b^\eps &\rats \nabla_{\x} p_b + \nabla_y p_b^1.
\end{align*}
Additional we assume that the pressure $p_b^1$ admits a separation of the form $p_b^1(\x,y) = p_{b,\Sigma}^1(\x) p_{b,N}^1(y)$ with $p_{b,\Sigma}^1 \in L^2(\Sigma)$ and $p_{b,N}^1 \in H^1(Y_f)$.

\item We have $g^\eps \in L^2(\Gamma_D^\eps)^n$ with $\eps^{-1} g^\eps \rats g^0$ for $g^0 \in L^2(\Sigma \times \Gamma_D)^n$ with $g^0(\x,y) = g_{\Sigma}(\x) g_{\Gamma_D}(y)$ and $g_{\Sigma} \in L^2(\Sigma)$, $g_{\Gamma_D} \in L^2(\Gamma_D)^n$. Especially, it holds that
\begin{align*}
    \Vert g^\eps \Vert_{L^2(\Gamma_D^\eps)} \le C\eps.
\end{align*}

\item \label{AssumptionGeometry} We assume that $\Gamma_D$ fulfills the following property: Let $\xi \in \R^n$ and $\xi \cdot \hat{n} = 0$ on $\Gamma_D$. Then we have $\xi = 0$.

\end{enumerate}

\begin{rem}\ 
\begin{enumerate}[label = (\roman*)]
    \item The assumption  \ref{AssumptionGeometry} is fulfilled iff $\Span\{\hat{n}(y) : y \in \Gamma_D\} = \R^n$, and is for example not true for a cylindrical inclusion $Y_s$. In this case we loose uniqueness for the microscopic model $\eqref{MicroModelNavierSlip}$. In fact, a weak solution is only unique up to a constant, see the proof of the Korn inequality in Lemma \ref{LemmaKornInequality}. However, uniqueness is guaranteed for other boundary conditions on a part of the lateral boundary, for example a no-slip condition or a pressure boundary condition of the form
\begin{align*}
    \hat{n}\cdot \left(p^\eps I-2\mu D(u^\eps)\right)\hat{n}& =p_{b}^\eps &\mbox{ on }& \Gamma_N^\eps \\
[\ueps]_{\tau}& =0 &\mbox{ on }& 
\Gamma_N^\eps.
\end{align*}
    
    \item The separation of the functions $p_b^1$ and $g^0$ is not necessary for the homogenization, but for a simpler structure of the Darcy-law (decoupling the macro- and the micro-variable). However, this assumption is  quite realistic, since $p_b^1$ is for example obtained from a homogenization process in a domain coupled to the top/bottom of the thin layer.
\end{enumerate}

\end{rem}

\begin{example}
As an example of admissible functions fulfilling the properties in the assumptions above we can take $f^{\epsilon} = f^0$ and $p_b^{\epsilon} = p_b$ (and therefore $p_b^1 = 0$) independent of the vertical variable. Further, for $g^{\epsilon}$ a possible choice is $g^{\epsilon}(x) = g_{\Gamma_D}\left(\frac{x}{\epsilon}\right)$, and therefore $g_{\Sigma} = 1$.
\end{example}

\section{Main results}\label{Main:sec}

We summarize the main results of our paper, where we have to distinguish between the cases $S_N^{\pm}$ empty or not, and also $\alpha = 0$ or positive. Since the structure of macroscopic models is quite similar and mainly differs by the boundary conditions on $S_N^{\pm}$ and $\Gamma_D$, we focus on the case $S_N^{\pm} = \emptyset$ and $\alpha = 0$, and formulate this macro model in detail. 
First of all, we introduce the function space
\begin{align}\label{DefinitionSpaceHn}
    \spaceHn &:=  \left\{ u \in H_{\#}^1(Y_f)^n \, : \, u\cdot \hat{n} = 0 \text{ on } \Gamma_D\right\},
\end{align}
the space of $(0,1)^{n-1}$-periodic functions with vanishing normal velocity on the solid surface $\Gamma_D$. We equip the space $\spaceHn$ with the norm 
\begin{align*}
    \Vert \phi \Vert_{\spaceHn} := \Vert \nabla_y \phi \Vert_{\spaceHn},
\end{align*}
which is in fact a norm, due to the Poincar\'e inequality in Lemma \ref{PoincareInequalityHn}.
\\

Now, let us introduce the two pressure homogenized Stokes model for $S_N^{\pm} = \emptyset$ and $\alpha = 0$, which is given by 
\begin{subequations}\label{TwoScaleMacroModelAlpha0}
\begin{align}
    -2\mu \nabla_y \cdot (D_y(u_0)) + \nabla_{\x} p_0 + \nabla_y p_1 &= f^0 &\mbox{ in }& \Sigma \times Y_f,
    \\
    \nabla_y \cdot u_0 &= 0 &\mbox{ in }& \Sigma \times Y_f,
    \\
    \nabla_{\x} \cdot \int_{Y_f} u_0 dy &= 0 &\mbox{ in }& \Sigma,
    \\
    \label{TwoScaleMacroModelNormalFluxZeroGammaD}
    u_0 \cdot \hat{n} &= 0 &\mbox{ on }& \Sigma \times \Gamma_D,
    \\
    \label{TwoScaleMacroModelStressBC}
    2\mu(D_y(u_0)\hat{n})_{\tau} &= g_0 &\mbox{ on }& \Sigma \times \Gamma_D,
    \\
    p_0 &= p_b &\mbox{ on }& \partial \Sigma,
    \\
    u_0 &\mbox{ is } (0,1)^{n-1}\mbox{-periodic}.
\end{align}
\end{subequations}
We say that the triple $(u_0,p_0,p_1)$ is a weak solution of problem $\eqref{TwoScaleMacroModelAlpha0}$ if
\begin{align*}
    (u_0,p_0,p_1) \in L^2(\Sigma,H_{\#}^1(Y_f)^n) \times  H^1(\Sigma) \times L^2(\Sigma \times Y_f)/\R
\end{align*}
and for all $\phi \in L^2(\Sigma ,\spaceHn)$ it holds that
\begin{align}
\begin{aligned}
\label{VariationalEquationTwoPressureProblem}
        2\mu \int_{\Sigma}\int_{Y_f}& D_y(u_0) : D_y(\phi) dy d\x + \int_{\Sigma} \int_{Y_f } \nabla_{\x} p_0\cdot \phi dy d\x - \int_{\Sigma } \int_{Y_f } p_1 \nabla_y \cdot \phi dy d\x
    \\
    &= \int_{\Sigma } \int_{Y_f } f^0 \cdot \phi dy d\x  + \int_{\Sigma} \int_{\Gamma_D } g^0 \cdot \phi d\sigma_y d\x.
\end{aligned}
\end{align}
\begin{thm}[$\alpha = 0$ and $S_N^{\pm} = \emptyset$]\label{MainResultAlphaZeroSNempty}
Let $(\ueps,\peps)$ be the microscopic solution of $\eqref{MicroModelNavierSlip}$ for $\alpha = 0$ and $S_N^{\pm} = \emptyset$. Then the sequence $\left(\epsilon^{-2} \ueps , \peps\right)$ converges in the two-scale sense to a limit function $(u_0,p_0) \in L^2(\Sigma, H_{\#}^1(Y_f)^n) \times H^1(\Sigma)$, and there exists $p_1 \in L^2(\Sigma \times Y_f)/\R$, such that the triple $(u_0,p_0,p_1)$ is the unique weak solution of the  
two pressure homogenized Stokes problem $\eqref{TwoScaleMacroModelAlpha0}$.

Further, $p_0$ is the unique weak solution of the following Darcy-law:
\begin{align}
\begin{aligned} \label{DarcyLawAlphaZero}
-\frac{1}{\mu}\nabla_{\x} \cdot \left(K(f^0 - \nabla_{\x} p_0\right) &= - \frac{1}{\mu} \nabla_{\x} \cdot \left(\beta g_{\Sigma} \right) &\mbox{ in }& \Sigma,
\\
p_0 &= p_b &\mbox{ on }& \partial \Sigma,
\end{aligned}
\end{align}
where the effective tensors $K$ and $\beta$ are defined in $\eqref{DefinitionEffectiveTensors}$ via cell problems.
\end{thm}
The two-scale convergence  is shown in Lemma \ref{ConvergenceResults} and the macroscopic model is derived in Section \ref{SectionCaseAlphaZeroSNZero}.
\begin{rem}
The Darcy velocity is given by the averaged quantity $\bar{u}_0 (\x) := \int_{Y_f} u_0(\x,y) dy $ for almost every $\x \in \Sigma$ and fulfills
\begin{align}
    \bar{u}_0(\x) = \frac{1}{\mu} K(f^0(\x) - \nabla_{\x} p_0(\x)) +\frac{1}{\mu} \beta g_{\Sigma}(\x).
\end{align}
\end{rem}
For the case $\alpha = 0$ and $S_N^{\pm} \neq \emptyset$ we have to add an additional boundary term including the first order pressure corrector $p_b^1$, see assumption \ref{AssumptionPressure}. 

\begin{thm}[$\alpha = 0$ and $S_N^{\pm} \neq \emptyset$]\label{MainResultAlphaZeroSNnotempty}
The microscopic solution $(\ueps,\peps)$ of $\eqref{MicroModelNavierSlip}$ for $\alpha = 0$ and $S_N^{\pm} \neq \emptyset$ fulfills the same compactness result as in Theorem \ref{MainResultAlphaZeroSNempty}. Further, there exists $p_1 \in L^2(\Sigma \times Y_f)$ (unique, not only up to a constant), such that the triple $(u_0,p_0,p_1)$ is the unique weak solution of the  
two pressure homogenized Stokes problem $\eqref{TwoScaleMacroModelAlpha0}$ with the additional boundary condition 
\begin{align*}
    \left[ p_1 I - 2\mu D_y(u_0)\right]\hat{n} = p_b^1 \hat{n} \qquad\mbox{on } \Sigma \times  S_N^{\pm}.
\end{align*}

Additionally, $p_0$ is the unique weak solution of the following Darcy-law:
\begin{align}
\begin{aligned} \label{DarcyLawAlphaZeroSNneqempty}
-\frac{1}{\mu}\nabla_{\x} \cdot \left(K(f - \nabla_{\x} p_0)\right) &= - \frac{1}{\mu} \nabla_{\x} \cdot \left(\beta g_{\Sigma} + \kappa p_{b,\Sigma}^1 \right) &\mbox{ in }& \Sigma,
\\
p_0 &= p_b &\mbox{ on }& \partial \Sigma,
\end{aligned}
\end{align}
where the effective tensors $K$, $\kappa$, and $\beta$ are defined in Section \ref{SectionAlphaZeroSNneqempty}.
\end{thm}

For $\alpha >0 $ we additionally have to distinguish between the three cases $\gamma < -1$, $\gamma =-1$, and $\gamma>-1$. 
Only in the critical case $\gamma =-1$ an additional contribution in the macroscopic model occurs on the boundary $\Gamma_D$.  For $\gamma < -1$ the slip conditition becomes a no-slip condition.
\begin{thm}\label{MainResultAlphaNotZero}
Let $\alpha  > 0$. Then we have:
\begin{enumerate}[label = (\roman*)]
    \item\label{MainResultAlphaNotZeroGammaBigger-1} For $\gamma>-1$ we obtain the same results as in Theorem \ref{MainResultAlphaZeroSNempty} and \ref{MainResultAlphaZeroSNnotempty} for $S_N^{\pm}=\emptyset$ resp. $S_N^{\pm} \neq \emptyset$.
    \item\label{MainResultAlphaNotZeroGammaEqual-1} For $\gamma = -1$ we have to replace in the two pressure Stokes system the boundary condition on $\Gamma_D$ with
    \begin{align*}
    2\mu(D_y(u_0)\hat{n})_{\tau} + \alpha u_0 &= g_0 &\mbox{ on }& \Sigma \times \Gamma_D.
    \end{align*}
    Then we obtain again the results from Theorem \ref{MainResultAlphaZeroSNempty} and \ref{MainResultAlphaZeroSNnotempty} with modified boundary condition on $\Gamma_D$, also for the cell problems  $\eqref{CellProblemWi}$, $\eqref{CellProblemWGamma}$, and $\eqref{CellProblemWN}$ for $S_N^{\pm} \neq \emptyset$.
    \item\label{MainResultAlphaNotZeroGammaSmaller-1} For $\gamma < -1$ we have to replace the boundary conditions $\eqref{TwoScaleMacroModelNormalFluxZeroGammaD}$ and $\eqref{TwoScaleMacroModelStressBC}$ on $\Gamma_D$ in the two pressure Stokes system by the no slip condition
    \begin{align*}
        u_0 = 0 \qquad\mbox{on } \Sigma \times \Gamma_D.
    \end{align*}
    Again we obtain similar results as in Theorem \ref{MainResultAlphaZeroSNempty} and \ref{MainResultAlphaZeroSNnotempty} with modified boundary conditions in the cell problems, if we put (formally) $g_{\Sigma}$ and $g_{\Gamma_D}$ equal to zero.
\end{enumerate}

\end{thm}

\begin{rem}
For $\gamma < -1$ the boundary force $g_0$ has no influence on the macroscopic equation, and we only expect a contribution in higher order corrector terms.
\end{rem}

\section{Existence, \textit{a priori} estimates, and two-scale compactness}\label{SectionExistenceApriori}

To show existence of a weak solution and uniform \textit{a priori} estimates with respect to $\eps$ we need a Korn inequality for functions with zero-normal trace. In general, for arbitrary Lipschitz domains $U \subset \R^n$,  $\Vert D(\cdot)\Vert_{L^2(U)}$ is not a norm on the space $H^1(U)^n$ with vanishing normal traces (at least on a part of the boundary $\partial U$). For $U\subset \R^n$ Lipschitz domain, and $\Lambda \subset \partial U$ with positive measure we introduce the space (see also \cite{amrouche2014lp})
\begin{align*}
    \mathcal{T}_{\Lambda}(U) := \left\{ u \in H^1(U)^n \, : \, D(u) = 0, \, \, u\cdot \hat{n} = 0 \text{ on } \Lambda \right\}.
\end{align*}
In \cite[Lemma 3.3]{amrouche2014lp} the following Korn inequality is proved: For all $u \in H^1(U)^n$ it holds that
\begin{align*}
    \mathrm{inf}_{v \in \mathcal{T}_{\partial U}(U) } \Vert u + v \Vert_{L^2(U)}^2 \le C \left(\Vert D(u) \Vert_{L^2(U)}^2 + \int_{\partial U} \vert u \cdot \hat{n} \vert^2 d\sigma \right).
\end{align*}
It is easy to check that the proof can be transferred to the case when we replace $\partial U$ with $\Lambda$. Hence, we obtain for all $u \in H^1(U)^n$
\begin{align*}
   \mathrm{inf}_{v \in \mathcal{T}_{\Lambda}(U)} \Vert u +v \Vert_{L^2(U)}^2 \le C \left(\Vert D(u) \Vert_{L^2(U)}^2 + \int_{\Lambda} \vert u \cdot \hat{n} \vert^2 d\sigma \right).
\end{align*}
Especially, if $U$ and $\Lambda$ are such that $\mathcal{T}_{\Lambda}(U) = \emptyset$, it holds for  all $u \in H^1(U)^n$ with $u \cdot \hat{n} = 0$ on $\Lambda$ that
\begin{align*}
    \Vert u \Vert_{L^2(U)} \le C \Vert D(u)\Vert_{L^2(U)}.
\end{align*}
In the following we investigate how this result is valid for our geometrical setting. 
First, we introduce the set
\begin{align*}
    \mathcal{E}:= \{-1,0,1\}^{n-1} \times \{0\},
\end{align*}
Now we define for $\ast =f,s$ the set
\begin{align*}
    \mathcal{Z}_{\ast}:= \left\{ \hat{Y}_{\ast}^I = \mathrm{int}\bigcup_{k \in I} \overline{Y_{\ast}} + k \, : \, I \subset \mathcal{E} , \, 0 \in I, \, \mathrm{dim}\left(\mathrm{span}(I)\right) = n-1\right\}. 
\end{align*}
The set $\mathcal{Z}_{\ast}$ contains sets including the reference element $Y_{\ast}$ and some neighboring cells. We emphasize that for every microscopic cell $\epsilon (Y_f + k) \subset \Omega^{\epsilon}$ with $k \in \Z^{n-1} \times \{0\}$, there exists $\hat{Y}_f^I \in \mathcal{Z}_f$, such that $\epsilon (Y_f + k ) \subset \epsilon (\hat{Y}_f^I + k) \subset \Omega^{\epsilon}$. Obviously, for every inner cell $\epsilon (Y_f + k)$ not touching the lateral boundary $\partial \Sigma \times (-\epsilon,\epsilon)$, we can choose $I= \mathcal{E}$. However, this is not true for microscopic cells touching the lateral boundary, and therefore the more complicated definition of $\mathcal{Z}_{\ast}$ is needed. Let us define the for elements $\hat{Y}_f^I$ and $\hat{Y}_s^I$ the inner boundary
\begin{align*}
    \hat{\Gamma}_D^I := \mathrm{int} \left(\overline{\hat{Y}_f } \cap \overline{\hat{Y}_s}\right).
\end{align*}
We have the following local Korn-inequalities:
\begin{lem}\label{LemmaLocalKornInequality}\ 
\begin{enumerate}[label = (\roman*)]
\item Let $\hat{Y}_f^I \in \mathcal{Z}_f$. For every $u \in H^1(\hat{Y}_f^I)^n$ with $u \cdot \hat{n}=0$ on $\hat{\Gamma}_D^I$, it holds that
\begin{align}\label{LocalKornInequality}
    \Vert u \Vert_{L^2(\hat{Y}_f^I)} \le C \Vert D(u) \Vert_{L^2(\hat{Y}_f^I)}.
\end{align}
\item For every $u \in H_{\#}^1(Y_f)^n$ with $u \cdot \hat{n} = 0$ on $\Gamma$ inequality $\eqref{LocalKornInequality}$ is also valid (if we replace $\hat{Y}_f^I$ with $Y_f$).
\end{enumerate}
\end{lem}
\begin{proof}
The first part of the lemma follows, if we show that $\mathcal{T}_{\hat{\Gamma}_D^I}(\hat{Y}_f^I) = \{0\}$.  Assume that $v \in \mathcal{T}_{\hat{\Gamma}_D}(\hat{Y}_f^I) $. 
Since $D(v) = 0$, it follows that $v$ is a rigid displacement, i.e., there exist $\alpha \in \R^n$ and $A\in \R^{n\times n}$ with $A^T = -A$, such that $v(y) = \alpha + Ay$. For every $e \in I$ and $y \in \Gamma_D$ it holds that $\hat{n} (y+e) = \hat{n}(y)$ and 
\begin{align*}
    0 = v(y+e) \cdot \hat{n}(y+e) = \underbrace{v(y)\cdot \hat{n}(y)}_{=0} + Ae \cdot \hat{n}(y)  = Ae \cdot \hat{n}(y).
\end{align*}
This implies with assumption \ref{AssumptionGeometry}
$Ae = 0$ for all $e \in I$. Since $\mathrm{span}(I)$ has dimension $n-1$ and the rank of $A$ is even ($A$ is skew-symmetric), we obtain $A=0$. Hence, we have $\alpha \cdot \hat{n} = 0$ on $\Gamma_D$, what implies $\alpha = 0$.
\\
The second statement follows easily,  since for every $u \in H_{\#}^1(Y_f)^n$ with $D(u)=0$ we have $u = \alpha \in \R^n$, and if additionally $u \cdot \hat{n} = 0$ on $\Gamma_D$ we obtain as above that $u=0$. Now, we can argue as in the proof of \cite[Lemma 3.3]{amrouche2014lp}.
\end{proof}

\begin{rem}\ 
\begin{enumerate}[label = (\roman*)]
\item Inequality $\eqref{LocalKornInequality}$ is in general not valid for $u \in H^1(Y_f)^n$. For example take $n=3$, $Y_s $ equal a ball with center $m$ strictly included in $Y$. Then for arbitrary $\beta \in \R^3$ the function 
\begin{align*}
   y \mapsto  \beta \times (y - m)
\end{align*}
is an element of $\mathcal{T}_{\Gamma_D}(Y_f)$ and inequality $\eqref{LocalKornInequality}$ does not hold. For more discussions about the space $\mathcal{T}_{\partial U} (U)$ we refer to \cite{amrouche2014lp}.
\item Lemma \ref{LemmaLocalKornInequality} is valid if $Y_s$ is strictly included in $Y$, since in this case assumption  \ref{AssumptionGeometry} is fulfilled.
\end{enumerate}
\end{rem}
Using the usual Korn inequality, we obtain for $\hat{Y}_f^I\in \mathcal{Z}_f$ and for all $u \in H^1(\hat{Y}_f^I)^n $  with $u \cdot \hat{n} = 0$ on $\hat{Gamma}^I_D$ that
\begin{align}\label{KornInequalityZeroNormalTrace}
    \Vert u \Vert_{H^1(\hat{Y}_f^I)} \le C \Vert D(u) \Vert_{L^2(\hat{Y}_f^I)}.
\end{align}
As a direct consequence we obtain by a simple decomposition argument for $\Omega^\eps$.
\begin{lem}\label{LemmaKornInequality}
For every $v^\eps \in H^1(\Omega^\eps)^n$ with $v^\eps \cdot \hat{n} =0 $ on $\Gamma_D^\eps$ it holds that
\begin{align*}
    \Vert v^\eps \Vert_{L^2(\Omega^\eps)} + \eps \Vert \nabla v^\eps \Vert_{L^2(\Omega^\eps)} \le C\eps \Vert D(v^\eps)\Vert_{L^2(\Omega^\eps)}.
\end{align*}
\end{lem}
We also have the following well-known trace inequality for all $v^\eps \in H^1(\Omega^\eps)$:
\begin{align}
    \label{TraceInequality}
    \sqrt{\eps} \Vert v^\eps \Vert_{L^2(\partial \Omega^\eps)} \le C \left( \Vert v^\eps \Vert_{L^2(\Omega^\eps)} + \epsilon \Vert \nabla v^\eps\Vert_{L^2(\Omega^\eps)} \right).
\end{align}
Now, we start with the estimates for the fluid velocity $u^\eps$.
\begin{lem}
\label{AprioriEstimatesVelocity}
The microscopic fluid velocity $u^\eps$ fulfills  for $\alpha > 0$
\begin{align*}
    \frac{1}{\sqrt{\eps}} \Vert u^\eps \Vert_{L^2(\Omega^\eps)} + \sqrt{\eps} \Vert \nabla u^\eps \Vert_{L^2(\Omega^\eps)}  &\le C\eps^2,
    \\
    \Vert u^\eps \Vert_{L^2(\Gamma_D^\eps)} &\le C \min\left\{\eps^2, \eps^{\frac{3-\gamma }{2}} \right\} .
\end{align*}
For $\alpha = 0$ it holds that
\begin{align*}
    \Vert u^\eps \Vert_{L^2(\Gamma_D^\eps)} &\le C \eps^2.
\end{align*}
\end{lem}
\begin{proof}
We test the variational equations $\eqref{MicroModelNavierSlipVarEqu}$ with $u^\eps$ to obtain with the assumptions on $f^\eps$, $g^\eps$, and $p_b^\eps$
\begin{align*}
    2 \mu \Vert & D(u^\eps)\Vert_{L^2(\Omega^\eps)}^2 + \alpha \epsilon^\gamma \Vert u^\eps\Vert_{L^2(\Gamma_D^\eps)}^2 
    \\
    &= \int_{\Omega^\eps} (f^\eps - \nabla p_b^\eps) \cdot u^\eps dx + \int_{\Gamma_D^\eps } g^\eps \cdot u^\eps d\sigma 
    \\
    &\le \left( \Vert f^\eps \Vert_{L^2(\Omega^\eps)} + \Vert \nabla p_b^\eps \Vert_{L^2(\Omega^\eps)} \right) \Vert u^\eps \Vert_{L^2(\Omega^\eps)} + \Vert g^\eps \Vert_{L^2(\Gamma_D^\eps)} \Vert u^\eps \Vert_{L^2(\Gamma_D^\eps)}
    \\
    &\le C \sqrt{\eps} \Vert u^\eps \Vert_{L^2(\Omega^\eps)} + C \eps \Vert u^\eps \Vert_{L^2(\Gamma_D^\eps)}.
\end{align*}
Using the Korn inequality from Lemma \ref{LemmaKornInequality} and the trace inequality in $\eqref{TraceInequality}$, we obtain
\begin{align*}
     \mu \Vert & D(u^\eps)\Vert_{L^2(\Omega^\eps)}^2 + \alpha \epsilon^\gamma \Vert u^\eps\Vert_{L^2(\Gamma_D^\eps)}^2 \le C\epsilon^3.
\end{align*}
Using again the Korn inequality we get
\begin{align*}
   \frac{1}{\sqrt{\eps}} \Vert u^\eps \Vert_{L^2(\Omega^\eps)} + \sqrt{\eps} \Vert \nabla u^\eps \Vert_{L^2(\Omega^\eps)} + \eps^{\frac{\gamma +1}{2}} \Vert u^\eps \Vert_{L^2(\Gamma_D^\eps)} \le C \epsilon^2.
\end{align*}
The trace inequality $\eqref{TraceInequality}$ now implies
\begin{align*}
    \Vert u^\eps \Vert_{L^2(\Gamma_D^\eps)} \le C \eps^2.
\end{align*}
\end{proof}

It remains to estimate the pressure $p^\eps$. For this we make use of the restriction operator $R_\eps$ which was first introduced by Tartar in \cite[Appendix]{SanchezPalencia1980} for perforated domains with strictly included obstacles, and later extended in \cite{allaire1989homogenization} to domains with connected solid phase. Here, we define a restriction operator for thin layers and give a slightly different proof than in \cite{allaire1989homogenization}.


\begin{prop}\label{RestrictionOperatorGlobal}
There exists a linear operator
\[
R_\eps\colon H^1(\Omega^{\epsilon}_M,S_\eps^{\pm})^n \to H^1(\Omega^\eps, \Gamma_D^\eps)^n
\]
such that for all $v\in H^1(\Omega^{\epsilon}_M,S_\eps^{\pm})^n,$ it holds that
\begin{enumerate}[label = (\roman*)]
    \item $\nabla\cdot (R_\eps v) =\nabla\cdot v$ in $\Omega^\eps$ if $v=0$ in $\Omega^\eps_M -\Omega^\eps,$
    \item $\nabla \cdot v = 0$ in $\Omega^\eps_M$ implies $\nabla \cdot R_\eps v = 0$ in $\Omega^\eps$,
    
    \item It holds that
    \begin{align*}
        \Vert R_\eps v \Vert_{L^2(\Omega^\eps)} + \epsilon \Vert \nabla R_\eps v \Vert_{L^2(\Omega^\eps)} \le C\eps \Vert \nabla v \Vert_{L^2(\Omega^\eps_M)}.
    \end{align*}
\end{enumerate}

\end{prop}

\begin{proof}
 As in [Tar80] the restrictions are constructed locally. Moreover, some aspects of the extended result in \cite{allaire1989homogenization} have been simplified.

  Let $\{S_i\}_{i\in I},$ $I=\{\pm 1,\pm 2\, \ldots , \pm (n-1) \},$ be an enumeration of the lateral faces of the  cell $Y$ so that $S_i$ is opposite to $S_{-i}$ and with outward unit normal $\hat{n}_i.$ Subordinate to $\{S_{\pm i}\}_{i\in I}$ we choose a family of functions $\{\phi_{\pm i}\}_{i\in I}$ such that for each $i\in I,$
 \begin{itemize}
\item $\phi_i\in C^\infty(\overline{Y}),$ $\phi_i=0$ in $Y_s$
\item $\phi_i=0$ on $S_j,$ whenever $i\neq j,$
\item $\phi_i\rvert_{S_i}=\phi_{-i}\circ \tau\rvert_{S_i},$ where $\tau$ is the canonical translation that maps $S_i$ to $S_{-i},$
\item $\int_{S_i} \phi_i\,d\sigma=1.$
\end{itemize}

A local restriction of any $v\in H^1(Y,S^{\pm})^n$ is constructed as follows. Let $\hat{v}$ be the vector field defined by
\begin{equation}
\hat{v}(y)=\sum_{i\in I} \left(\int_{S_i} v\cdot \hat{n}\,d\sigma\right)\phi_i(y)\hat{n}_i,\quad (y\in Y)
\end{equation}
Then $\hat{v}\in C^\infty(\overline{Y})^n,$ $\hat{v}=0$ in $Y_s$ and since the outward unit normal $\hat{n}_i$ is a constant vector, $\hat{v}$ preserves the outward flux of $v$ through each lateral face of $Y,$ i.e.
\begin{equation}\label{vflux:eq}
\int_{S_i} \hat{v}\cdot \hat{n}\,d\sigma=\int_{S_i} v\cdot \hat{n}\,d\sigma\quad\text{($i=1,\dotsc,n-1$)}.
\end{equation}
In general, $\hat{v}$ does not  preserve the divergence of $v$ in $Y_f.$ However, from \eqref{vflux:eq} and the divergence theorem, we deduce
\[
\int_{Y_{f}} \nabla\cdot\hat{v}\,dx=\int_{Y_{f}} \nabla\cdot v\,dx+\int_{Y_{s}} \nabla\cdot v\,dx.
\]
Here we used that $v = 0$ on $S^{\pm}$.
Set
\[
f=\nabla\cdot v+\frac{1}{\lvert{Y_{f}} \rvert  }  \left(\int_{Y_{s}}\nabla\cdot v\,dx\right)-\nabla\cdot\hat{v}.
\]
From the Bogovski\u{i} theorem, see the Appendix \ref{SectionAppendixBogovskii},
the boundary value problem
\[
\left\{
\begin{aligned}
\nabla\cdot w& =f&& \text{in }Y_f\\
w& =0 && \text{on }\partial Y_f
\end{aligned}\right.
\]
has a solution $w =Bf \in H_0^1(Y_f)^n$, which we extend by zero to the whole cell $Y$ and fulfills
\begin{align*}
    \Vert \nabla w \Vert_{L^2(Y_f)} \le C \Vert f\Vert_{L^2(Y_f)}.
\end{align*}
Now we define the restriction as
\begin{equation}
Rv=\hat{v}+w.   
\end{equation}
It is readily verified that $Rv$ depends linearly on $v$ and has the following properties:
\begin{itemize}
    \item[(i)] $Rv=0$ in $Y_s$
    \item[(ii)] $\displaystyle\nabla\cdot(Rv)=\nabla\cdot{v}+\frac{1}{\lvert Y_{f}\rvert}\int_{Y_{s}}\nabla\cdot{v}\,dx$ 
    \item[(iii)] $Rv=\hat{v}$ on $\partial Y$
    \item[(iv)] There exists a constant $C$ depending only on $Y$ and $Y_s$ such that
    \[
    \rVert Rv\rVert_{H^1(Y_f)}\leq C\rVert \nabla v\rVert_{L^2(Y)}.
    \]
\end{itemize}

By scaling and translating the cell $Y$ we obtain a family of cells
\[
Y^{\eps,k}=\{y\in \mathbb{R}^n: y=\eps(z+k),\:z\in Y\},\quad k\in \mathbb{Z}^{n-1}\times \{0\}.
\]
For $Y^{\eps,k}$ contained in $\Omega_M^\eps$ we can define a family of local restrictions
\[
R_{\eps,k}\colon H^1(\Omega_M^\eps,S_\eps^\pm)^n\to H^1(Y^{\eps,k},\partial Y_s^{\eps,k})^n,
\]
according to the scheme described above.  The property (iii) ensures that the boundary values of restricted functions are compatible on the faces of adjacent cells. Thus, after a finite number of restrictions, we obtain a global restriction operator $R_\eps$ from $H^1(\Omega_M^\eps)^n$ to $H^1(\Omega^\eps,\Gamma_D^\eps)^n$ that satisfies the properties (i)-(iii).

\end{proof}

\begin{rem}\ 
\begin{enumerate}[label = (\roman*)] 
    \item We emphasize that, in contrast to the restriction operator constructed in \cite{allaire1989homogenization}, our restriction operator does not fulfill $R_{\epsilon} v = v$ for $v \in H^1(\Omega_M^{\epsilon},S_{\epsilon}^{\pm})$ with $v=0 $ in $\Omega_M^{\epsilon} - \Omega^{\epsilon}$, but only the weaker condition $\nabla \cdot R_{\epsilon} v = \nabla \cdot v$, which simplifies the proof. This is enough to extend the pressure to the whole layer and obtain uniform \textit{a priori} bounds in $L^2$. 
    \item The restriction operator is not adapted to the zero normal flux boundary conditions on $\Gamma_D^{\epsilon}$. In fact, it might be more intuitive with respect to the micro-model to define the restriction operator in such way that $R_{\epsilon } v \cdot \hat{n} = 0$ on $\Gamma_D^{\epsilon}$. This is not necessary since we define the extension of the pressure as a gradient in the distributional sense.
\end{enumerate}

\end{rem}

In the following for a Banach space $X$ we denote the duality pairing $\langle \cdot , \cdot \rangle_{X',X}$ with its dual space $X'$ shortly by $\langle \cdot , \cdot \rangle_X$. Lemma \ref{LemmaBasicTSCompactness} implies that the divergence operator (we neglect its dependence on the domain of definition)
\begin{align*}
    \diver : H^1(\Omega^\eps_M , S_\eps^\pm )^n \rightarrow L^2(\Omega^\eps_M)
\end{align*}
is surjective  and for every $\epsilon$ there exists a constant $C_\eps $ such that for every $P^\eps \in L^2(\Omega^\eps_M)$ there exists  $v^\eps \in H^1(\Omega^\eps_M,S_\eps^\pm)^n$ such that
\begin{align*}
   \epsilon \Vert \nabla v^\eps \Vert_{L^2(\Omega^\eps_M)} \le C_\eps \Vert P^\eps \Vert_{L^2(\Omega^\eps_M)}.
\end{align*}
We emphasize that by the Poincar\'e inequality  the left-hand side defines a norm on $H^1(\Omega^\eps_M)^n$. More precisely, it holds that
\begin{align*}
    \Vert v^\eps \Vert_{L^2(\Omega^\eps_M)}  \le C \epsilon \Vert \nabla v^\eps \Vert_{L^2(\Omega^\eps_M)}.
\end{align*}
By a simple scaling  for the $x_3$-component we can transform $\Omega^\eps_M$ to the fixed domain $\Omega^1_M$. Then it is easy to check that 
\begin{align*}
    C_\eps \le C C_1 \qquad \text{for all } \epsilon.
\end{align*}
Hence, we obtain for a constant $C $ independent of $\epsilon$
\begin{align}\label{InequalityBogovskii}
   \epsilon \Vert \nabla v^\eps \Vert_{L^2(\Omega^\eps_M)} \le C \Vert P^\eps \Vert_{L^2(\Omega^\eps_M)}.
\end{align}
We define the space $H_\eps$ as the space of functions in $H^1(\Omega^\eps_M , S_\eps^\pm )^n$ with the norm $\epsilon \Vert \nabla v^\eps \Vert_{L^2(\Omega^\eps_M)}$.
By the surjectivity of $\diver$ and the closed range theorem, we obtain for every $F^\eps \in H_\eps'$ 
with
\begin{align*}
    \langle F^\eps , v^\eps \rangle_{H_\eps} = 0 
\end{align*}
for all $v^\eps \in H^1(\Omega^\eps_M , S_\eps^\pm )^n$ with $\nabla \cdot v^\eps = 0$ the existence of $P^\eps \in L^2(\Omega^\eps_M)$ such that
\begin{align}\label{RepresentationFeps}
     \langle F^\eps , v^\eps \rangle_{H_\eps} = - \int_{\Omega^\eps_M} P^\eps \nabla \cdot v^\eps dx \qquad \mbox{ for all } v^\eps \in  H^1(\Omega^\eps_M , S_\eps^\pm )^n.
\end{align}
In other words, we have $F^{\epsilon} =- \diver^{\ast} P_{\epsilon}$, where $\diver^{\ast}$ denotes the adjoint of $\diver: H_\eps \rightarrow L^2(\Omega^\eps_M)$. Using inequality $\eqref{InequalityBogovskii}$, we easily obtain 
\begin{align}\label{InequalityL2Gradient}
    \Vert P^\eps \Vert_{L^2(\Omega^\eps_M)} \le C \Vert F^\eps \Vert_{H_\eps'}.
\end{align}
Now we are able to give an estimate for the microscopic pressure $p^\eps$:
\begin{lem}
\label{AprioriEstimatePressure}
There exists an extension $P^\eps \in L^2(\Omega^\eps_M)$ of $p^\eps - p_b^\eps$ such that
\begin{align*}
    \Vert P^\eps \Vert_{L^2(\Omega^\eps_M)} \le C \sqrt{\eps}.
\end{align*}
Especially, we obtain 
\begin{align*}
    \Vert p^\eps \Vert_{L^2(\Omega^\eps)} \le  C \sqrt{\eps}.
\end{align*}
Moreover
\[
P^\eps=
\left\{
\begin{aligned}
&p^\eps-p_b^\eps && \text{in }\Omega^\eps,\\
&\frac{1}{\lvert \eps Y_f\rvert}\int_{\eps(Y_f+k)} (p^\eps-p_b^\eps)\,dx && \text{in }\eps(Y_s+k),
\end{aligned}\right.
\]
for all $k\in \mathbb{Z}^{n-1}\times \{0\}$ such that the cell $\eps(Y+k)$ is contained in $\Omega_M^\eps.$
\end{lem}
\begin{proof}
 We define the functional $F^\eps \in H_\eps'$ by 
 \begin{align*}
     \langle F^\eps , v^\eps \rangle_{H_\eps} := -(p^\eps - p_b^\eps , \nabla \cdot R_\eps v^\eps)_{L^2(\Omega^\eps)},
 \end{align*}
i.e., we have $ \langle F^\eps , v^\eps \rangle_{H_\eps} =  \langle \diver^{\ast} (p^\eps - p_b^\eps), R_\eps v^\eps)_{L^2(\Omega^\eps)}$. Hence, using the microscopic equation $\eqref{MicroModelNavierSlipVarEqu}$, we get for all $v^\eps \in H^1(\Omega^\eps_M,S_\eps^\pm)^n$ (using $R_\eps v^\eps = 0$ on $\Gamma_D^\eps$)
 \begin{align*}
     \langle F^\eps , v^\eps \rangle_{H_\eps} =& -2\mu \int_{\Omega^\eps } D(u^\eps): D(R_\eps v^\eps) dx + \int_{\Omega^\eps} (f^\eps - \nabla p_b^\eps) \cdot R_\eps v^\eps dx .
 \end{align*}
Using the \textit{a priori} estimates from Lemma \ref{AprioriEstimatesVelocity} for the fluid velocity, the assumptions on the data, and the properties of the restriction operator $R_\eps $ from Proposition \ref{RestrictionOperatorGlobal} we obtain
\begin{align}
    \vert \langle F^\eps , v^\eps \rangle_{H_\eps} \vert &\le C \eps^{\frac32} \Vert \nabla R_\eps v^\eps  \Vert_{L^2(\Omega^\eps)} + \sqrt{\eps} \Vert R_\eps v^\eps  \Vert_{L^2(\Omega^\eps)}
    \\
    &\le C \eps^{\frac32} \Vert \nabla v^\eps \Vert_{L^2(\Omega^\eps_M) } = C \sqrt{\eps}\Vert v^\eps \Vert_{H_\eps} .
\end{align}
Hence, we have 
\begin{align*}
    \Vert F^\eps \Vert_{H_\eps'} \le C \sqrt{\eps}.
\end{align*}
Further, for all $v^\eps \in H^1(\Omega^\eps_M,S_\eps^{\pm})^n$ with $\nabla \cdot v^\eps = 0$ we obtain from the microscopic variational equation $\eqref{MicroModelNavierSlipVarEqu}$ that $\langle F^\eps , v^\eps \rangle_{H_\eps} = 0$. Hence, there exists $P^\eps \in L^2(\Omega^\eps_M)$ which fulfills $\eqref{RepresentationFeps}$. Especially, we obtain using inequality $\eqref{InequalityL2Gradient}$ 
\begin{align*}
     \Vert P^\eps \Vert_{L^2(\Omega^\eps_M)} \le C \Vert F^\eps \Vert_{H_\eps'} \le  C \sqrt{\eps}.
\end{align*}
It remains to check  that $P^\eps$ is an extension of $p^\eps - p_b^\eps$. 
For every $\veps \in H^1(\Omega_M^\eps, S_{\epsilon}^{\pm})^n$ it holds that
\begin{align*}
    -(P^\eps , \nabla \cdot \veps )_{L^2(\Omega_M^\eps)} = \langle F^\eps , \veps \rangle_{H_\eps} &= -(p^\eps -p^\eps_b , \nabla \cdot R_\eps \veps )_{L^2(\Omega^\eps)}
\end{align*}
Since the divergence operator is onto, for any $f\in L^2(\epsilon(Y+k))$ extended by zero to the whole layer $\Omega_M^{\epsilon}$, there exists $v^\eps\in H^1(\Omega_M^\eps, S_\eps^\pm)$ such that 
\[
\nabla\cdot v^\eps = f \mbox{ in } \Omega_M^\eps.
\]
For such $v^\eps$ it holds that
\begin{align*}
\int_{\eps(Y+k)} P^\eps\nabla \cdot \veps \,dx& =\int_{{\eps(Y_f+k)}}(p^\eps -p^\eps_b)\nabla\cdot (R_\eps v^\eps)\,dx\\
& =\int_{{\eps(Y_f+k)}}(p^\eps -p^\eps_b)\left(\nabla\cdot v^\eps+\frac{1}{\lvert\eps Y_f\rvert}\int_{\eps(Y_s+k)}\nabla\cdot v^\eps\,dx\right)\,dx\\
& =\int_{{\eps(Y_f+k)}}(p^\eps -p^\eps_b)\nabla\cdot v^\eps\,dx\\
&\quad +\int_{\eps(Y_s+k)}\frac{1}{\lvert\eps Y_f\rvert}\left(\int_{{\eps(Y_f+k)}}(p^\eps -p^\eps_b)\,dx\right)\nabla\cdot v^\eps\,dx.
\end{align*}
This implies the desired extension property.

\end{proof}


With the uniform estimates for $\ueps$ and $\peps$ obtained above and the general two-scale compactness results summarized in the Appenix \ref{SectionTSConvergence}, we immediately obtain the following convergence result.

\begin{lem}\label{ConvergenceResults}
There exists  $(u_0, p_0) \in L^2(\Sigma,H_{\#}^1(Y_f))^n \times L^2(\Sigma)$ with $\nabla_y \cdot u_0 = 0$ and $\nabla_{\x} \cdot \int_{Y_f} u_0 dy = 0$, such that up to a subsequence
\begin{align*}
   \chi_{\Omega^{\epsilon}} \frac{u^\eps}{\eps^2} &\rats \chi_{Y_f}u_0,
    \\
   \chi_{\Omega^{\epsilon}} \eps \frac{\nabla u^\eps } {\eps^2} &\rats \chi_{Y_f}\nabla_y u_0,
    \\
   \chi_{\Omega^{\epsilon}} p^\eps &\rats \chi_{Y_f}p_0.
\end{align*}
Further it holds that up to a subsequence
\begin{align*}
    \frac{u^\eps \vert_{\Gamma_D^\eps}}{\eps^2} \rats u_0\vert_{\Gamma_D} \qquad\text{on } \Gamma_D^\eps.
\end{align*}
Additionally, it holds that $u_0 \cdot \hat{n} = 0$ on $\Gamma_D$.
For $\alpha > 0$ and  $\gamma < -1$  we have up to a subsequence
\begin{align*}
    \frac{u^\eps \vert_{\Gamma_D^\eps}}{\epsilon^2} \rats 0 \qquad\text{on } \Gamma_D^\eps.
\end{align*}
Especially, we obtain $u_0 = 0$ on $\Gamma_D$ for $\alpha > 0$ and $\gamma < -1$.
\end{lem}
\begin{proof}
 The convergences of $\eps^{-2} u^\eps$ and the gradients follow directly from  the \textit{a priori} estimates in Lemma \ref{AprioriEstimatesVelocity} and Lemma \ref{TwoScaleCompactnessPerforated} in the appendix. The equations $\nabla_y \cdot u_0 = 0$ and $\nabla_{\x} \cdot \int_{Y_f} u_0 dy = 0$ are quite standard, see \cite{Allaire_TwoScaleKonvergenz} for similar arguments. For the pressure we use the \textit{a priori} bound of $p^\eps$ from Lemma \ref{AprioriEstimatePressure} to obtain the existence of $p_0 \in L^2(\Sigma \times Y_f)$ such that up to subsequence
 \begin{align}
    \chi_{\Omega^\eps} p^\eps \rats \chi_{Y_f} p_0.
 \end{align}
 Let us check that $p_0 $ is independent of $y$. We test the equation $\eqref{MicroModelNavierSlipVarEqu}$ with $\phi^\eps(x):= \phi(\x) \psi\left(\dfrac{x}{\epsilon}\right)$ for $\phi \in C_0^{\infty}(\Sigma)$ and $\psi \in C_0^{\infty}(Y_f)^n$ to obtain for $\epsilon \to 0$ almost everywhere in $\Sigma$
\begin{align*}
    \int_{\Sigma} \int_{Y_f } p_0(x,y) \nabla_y \cdot \psi(y) dy = 0.
\end{align*}
This implies $p_0(\x,y) = p_0(\x)$.  Let us show the convergence of the traces. For $\gamma \geq -1 $ (and $\alpha \geq 0$) we have from Lemma \ref{AprioriEstimatesVelocity}
\begin{align*}
    \Vert u^\eps \Vert_{L^2(\Gamma_D^\eps)} \le C\eps^2.
\end{align*}
With Lemma \ref{LemmaBasicTSCompactness} we obtain the two-scale convergence of $\eps^{-2} u^\eps\vert_{\Gamma_D^\eps} $ to $u_0\vert_{\Gamma}$. Now, for $\alpha \neq 0 $ and $\gamma < -1$ we get with Lemma \ref{AprioriEstimatesVelocity}
\begin{align}
    \eps^{-2} \Vert u^\eps \Vert_{L^2(\Gamma_D^\eps)} \le C\eps^{-\frac{1 + \gamma}{2}} \overset{\eps \to 0}{\longrightarrow} 0,
\end{align}
since $0  < -\frac{1 +\gamma}{2} $.
\end{proof}

\section{Two pressure characterization}
\label{SectionTwoPressureDecomposition}

In this section we show that every functional  $v^{\ast} \in L^2(\Sigma,\spaceHn)'$ vanishing on a suitable subspace, including divergence free functions with respect to the micro-variable and vanishing divergence for the mean with respect to the macro-variable, can be decomposed into the sum of the gradients of two pressures $v^{\ast} = \nabla_{\x} p_0  + \nabla_y p_1$. This result is crucial to obtain the pressures in the two pressure Stokes system, see for example $\eqref{TwoScaleMacroModelAlpha0}$. Similar results for perforated domains (not thin) and a no-slip boundary condition on the interior oscillating surface can be found in \cite{allaire1997one} and \cite[Section 14]{chechkin2007homogenization}. Here, we give a detailed proof for thin perforated domains and the zero normal flux boundary condition.

We define the spaces (see $\eqref{DefinitionSpaceHn}$ for the definition of the space $\spaceHn$)
\begin{align*}
    \spaceHndiv&:= \left\{ u \in H_{\#}^1(Y_f)^n \, : \, \nabla_y \cdot u = 0 \text{ in } Y_f , \, u\cdot \hat{n} = 0 \text{ on } \Gamma_D\right\} \subset \spaceHn,
    \\
    \spaceV &    := \left\{ \phi \in L^2(\Sigma,\spaceHndiv) \, : \, \nabla_{\x} \cdot \int_{Y_f} \phi(\cdot_{\x},y) dy = 0 \text{ in } \Sigma \right\} \subset L^2(\Sigma, \spaceHn).
\end{align*}
First of all, let us show a Poincar\'e inequality on $\spaceHn$, which especially implies that $\Vert \cdot \Vert_{\spaceHn}$ defines a norm on $\spaceHn$ equivalent to the usual $H^1$-norm.
\begin{lem}\label{PoincareInequalityHn}
For all $\phi \in H^1_{\#}(Y_f)^n$ it holds that
\begin{align*}
    \Vert \phi \Vert_{L^2(Y_f)}^2 \le C \left( \Vert \nabla_y \phi \Vert_{L^2(Y_f)}^2 + \int_{\Gamma} \vert \phi \cdot \hat{n} \vert^2 d\sigma \right).
\end{align*}
Especially, we have for all $\phi \in \spaceHn$
\begin{align*}
    \Vert \phi \Vert_{L^2(Y_f)} \le C \Vert \nabla_y \phi \Vert_{L^2(Y_f)}.
\end{align*}
\end{lem}
\begin{proof}
 The proof is quite standard. However, for the sake of completeness and to illustrate where some geometrical assumptions on $Y_f$ come into play, we give some details. We argue by contradiction. Assume that there exists a sequence $(\phi_k) \subset H^1_{\#}(Y_f)^n$ such that
 \begin{align}\label{ContradictionAssumptionPoincare}
     1 = \Vert \phi_k \Vert_{L^2(Y_f)}^2 \geq k \left(\Vert \nabla_y \phi_k \Vert_{L^2(Y_f)}^2 + \int_{\Gamma} \vert \phi_k \cdot \hat{n} \vert^2 d\sigma  \right) .
 \end{align}
 Then we obtain for a $\phi \in H^1_{\#}(Y_f)^n$  (up to a subsequence) 
 \begin{align*}
     \phi_k &\rightharpoonup \phi &\mbox{ in }& H^1(Y_f)^n,
     \\
     \nabla_y \phi_k &\rightarrow 0 &\mbox{ in }& L^2(Y_f)^{n\times n},
     \\
     \phi_k &\rightarrow \phi &\mbox{ in }& L^2(Y_f)^n,
     \\
     \phi_k\vert_{\Gamma_D} &\rightarrow \phi\vert_{\Gamma_D} &\mbox{ in }& L^2(\Gamma_D)^n.
 \end{align*}
 Especially, we have $\phi = const$. Using again $\eqref{ContradictionAssumptionPoincare}$ and the strong convergence of the trace of $\phi_k$, we obtain
 \begin{align*}
     \int_{\Gamma_D} \vert \phi \cdot \hat{n} \vert^2 d\sigma = 0
 \end{align*}
 and therefore $\phi \cdot \hat{n} = 0$ on $\Gamma_D$. Due to assumption \ref{AssumptionGeometry} we have $\phi = 0$, which contradicts  $1 = \Vert \phi \Vert_{L^2(Y_f)}$.
\end{proof}

\begin{rem}\ 
\begin{enumerate}[label = (\roman*)]
\item A similar Poincar\'e inequality is used in the proof of \cite[Lemma 2.5]{allaire1991homogenization} without the additional assumption \ref{AssumptionGeometry} on the domain $Y_f$. However, in our situation this assumption is necessary, since otherwise we could have cylindrical inclusions in the layer. Nevertheless, the assumption \ref{AssumptionGeometry} is always fulfilled for strict inclusions $Y_s$.
\item The periodicity of the functions in Lemma \ref{PoincareInequalityHn} is not used.
\end{enumerate}
\end{rem}

We have the following characterization for the orthogonal complement of $\spaceV$. 

\begin{lem}\label{CharacterizationComplementW}
It holds that
\begin{align}\label{DecompositionV}
    \spaceV^\perp = \left\{\nabla_{\x} p_0 + \nabla_y p_1 \in L^2(\Sigma,\spaceHn)' \, : \,  p_0 \in H_0^1(\Sigma), \, p_1 \in L^2(\Sigma \times Y_f)/\R \right\}. 
\end{align}
Further, the decomposition on the right-hand side is unique.
\end{lem}

\begin{rem}
 We emphasize that the annihilator of $\spaceV$ is taken with respect to the space $L^2(\Sigma,\spaceHn)$ and for $p_0 \in H_0^1(\Sigma)$ and $p_1 \in L^2(\Sigma \times Y_f)/\R$ we have the following embeddings for the gradients into the space $L^2(\Sigma,\spaceHn)'$:
\begin{align*}
    \langle \nabla_{\x} p_0, \phi \rangle_{L^2(\Sigma,\spaceHn)', L^2(\Sigma,\spaceHn)}&= (\nabla_{\x} p_0, \phi)_{L^2(\Sigma \times Y_f)},
    \\
    \langle \nabla_y p_1 , \phi \rangle_{L^2(\Sigma,\spaceHn)', L^2(\Sigma,\spaceHn)}&= - (p_1, \nabla_y \cdot \phi )_{L^2(\Sigma \times Y_f)}.
\end{align*}
The right-hand sides obviously define functionals on $L^2(\Sigma,\spaceHn)$ (see also the Poincar\'e-inequaltiy in Lemma \ref{PoincareInequalityHn}). The pressure $p_1$ on the right-hand side in the decomposition $\eqref{DecompositionV}$ is unique only up to a constant.
\end{rem}
\begin{proof}[Proof of Lemma \ref{CharacterizationComplementW}]
 We describe the space $\spaceV$ as the intersection of $V_1$ and $V_2$ defined via
 \begin{align*}
     V_1 &:= L^2(\Sigma,\spaceHndiv) = \left\{ \phi \in L^2(\Sigma,\spaceHn) \, : \, \nabla_y \cdot \phi = 0\right\}
     \\
     V_2 &:= \left\{ \phi \in L^2(\Sigma, \spaceHn) \, : \, \nabla_{\x} \cdot \int_{Y_f } \phi dy = 0 \right\}.
 \end{align*}
 Obviously, we have $\spaceV = V_1 \cap  V_2$. Since $V^{\perp} = (V_1 \cap V_2)^{\perp} = V_1^{\perp} \oplus V_2^{\perp}$, it is enough to characterize the spaces $V_1^{\perp}$ and $V_2^{\perp}$. First, we show
 \begin{align*}
     V_1^{\perp} =  \left\{ \nabla_y \phi_1 \, : \, \phi_1 \in L^2(\Sigma \times Y_f)/\R\right\}.
 \end{align*}
 This follows from the closed range theorem. In fact, since 
 \begin{align*}
      H_{\#,0}^1(Y_f):= \left\{ u \in H_{\#}^1(Y_f)\, : \, u=0 \text{ on } \Gamma_D \right\} \subset \spaceHn,
 \end{align*}
we have that the operator 
 \begin{align*}
     \div_y : \spaceHn \rightarrow L^2_0(Y_f) := \left\{ f \in L^2(Y_f) \, : \, \int_{Y_f} f dy = 0 \right\}
 \end{align*}
 is surjective and therefore has closed range. Hence, its adjoint
 \begin{align}
     -\nabla_y : L^2(Y_f) \rightarrow \spaceHn', \qquad \langle -\nabla_y p , \phi \rangle_{\spaceHn',\spaceHn} = (p , \nabla_y \cdot \phi )_{L^2(Y_f)}
 \end{align}
 has closed range with 
 \begin{align*}
     R(-\nabla_y) = N(\div_y)^{\perp} = \left\{ F \in \spaceHn' \, : \, \langle F,\phi \rangle_{\spaceHn',\spaceHn} = 0 \, \mbox{ for all } \phi \in \spaceHndiv \right\}.
 \end{align*}
 Since $\x \in \Sigma$ only acts as a parameter, we obtain
 \begin{align}
     V_1^{\perp}  = L^2(\Sigma, R(-\nabla_y)).
 \end{align}
 Further, we have $N(-\nabla_y) = R(\div_y)^{\perp} = L_0^2(Y_f)^{\perp} = \R$,
 which gives the desired result for $V_1^{\perp}$.
 
 Next, we show that
 \begin{align*}
     V_2^{\perp} = \left\{ \nabla_{\x} p_0 \, : \, p_0 \in H_0^1(\Sigma)\right\}.
 \end{align*}
 First of all, for $i=1,\ldots, n-1$, let $h_i \in \spaceHn$ be the unique solution of the problem (upper index $i$ stands for the component of a vector)
 \begin{align}\label{AuxiliaryCellProblemHiWeak}
     (\nabla_y h_i , \nabla_y \phi)_{L^2(Y_f)} = \int_{Y_f} \phi^i  dy \qquad\mbox{for all } \phi \in \spaceHn.
 \end{align}
 In other words, $h_i$ is the unique weak solution of the problem
 \begin{align}
 \begin{aligned}\label{AuxiliaryCellProblemHi}
     -\Delta_y h_i &= e_i &\mbox{ in }& Y_f,
     \\
     h_i \cdot \hat{n} &= 0 &\mbox{ on }& \Gamma_D,
     \\
     [\nabla_y h_i \hat{n}]_{\tau} &= 0 &\mbox{ on }& \Gamma_D,
     \\
     h_i &\mbox{ is } (0,1)^{n-1}\mbox{-periodic}.
 \end{aligned}
 \end{align}
Due to the Poincar\'e inequality from Lemma \ref{PoincareInequalityHn}, this problem has a unique weak solution. Now, we define the matrix $B \in \R^{(n-1) \times (n-1)}$ via ($i,j=1,\ldots,n-1$)
 \begin{align*}
     B_{ij} := \int_{Y_f} \nabla_y h_i : \nabla_y h_j dy = \int_{Y_f } h_i^j dy = \int_{Y_f} h_j^i dy.
 \end{align*}
 Obviously, $B$ is symmetric. Let us show that it is also positive. The proof is quite similar to the proof of the positivity of the permeability tensor for the Darcy problem, see \cite[Chapter 7, Proposition 2.2]{SanchezPalencia1980}. However, we give some details. Let $\xi = \R^{n-1} \times \{0\}$ and $h:= \sum_{i=1}^n \xi_i h_i \in \spaceHn$. Then we have (with $ \bar{\xi} =(\xi_1,\ldots,\xi_{n-1})$)
 \begin{align*}
     B\bar{\xi} \cdot \bar{\xi} = \Vert \nabla_y h \Vert_{L^2(Y_f)}^2 = \Vert h \Vert_{\spaceHn}^2 \geq 0.
 \end{align*}
 In the case of equality we have $h=0$. Then we have for all $\phi \in \spaceHn$
 \begin{align}
     0 = (\nabla_y h , \nabla_y \phi )_{L^2(Y_f)} = \xi \cdot \int_{Y_f} \phi dy.
 \end{align}
 Choosing $\phi$ in such a way that $\int_{Y_f} \phi dy = \xi$, we obtain $\xi = 0$. This gives the positivity of $B$. 
 
 Now, we define the operator (see \cite[Lemma 2.10]{Allaire_TwoScaleKonvergenz} for a similar construction)
 \begin{align*}
     F: L^2(\Sigma)^{n-1} \rightarrow L^2(\Sigma,\spaceHn), \qquad F(\theta)(\x,y) = \sum_{i=1}^{n-1} \left(B^{-1}\theta(\x)\right)^i h_i.
 \end{align*}
 We use the short notation $\phi_{\theta} := F(\theta)$. The operator $F$ is linear and bounded, and it holds for $i=1,\ldots,n-1$ that
 \begin{align}\label{MeanValueF}
     \int_{Y_f} \phi_{\theta}^i dy = \theta^i.
 \end{align}
 We emphasize that the $n$-th component in the construction of $\phi_{\theta}$ has no influence. In fact, it is easy to check that a functional in $V_2^{\perp}$ is not depending on the $n$-th component of the test-function.
  Using the positivity of $B^{-1}$, we obtain
 \begin{align*}
     \Vert \phi_{\theta} \Vert_{L^2(\Sigma,\spaceHn)}^2 &=\sum_{i,j=1}^{n-1} \int_{\Sigma} \left(B^{-1} \theta(\x)\right)^i \left(B^{-1}\theta(\x)\right)^j \underbrace{(\nabla_y h_i ,\nabla_y h_j)_{L^2(Y_f)}}_{=B_{ij}} d\x 
     \\
     &= \sum_{i,j,k,l=1}^{n-1} \int_{\Sigma} B_{ik}^{-1} B_{jl}^{-1} \theta^k(\x) \theta^l(\x) B_{ij} d\x 
     \\
    &= \int_{\Sigma} B^{-1}\theta(\x) \cdot \theta(\x) d\x \geq c_0 \Vert \theta\Vert_{L^2(\Sigma)}^2.
 \end{align*}
 Hence, the operator $F$ has closed range $R(F)$. Next, we prove that the orthogonal complement of $R(F)$ in $L^2(\Sigma,\spaceHn)$ fulfills
 \begin{align}\label{CharacterizationRF}
     R(F)^{\perp} = \left\{ \psi \in L^2(\Sigma, \spaceHn) \, : \, \int_{Y_f} \psi^i dy = 0 \mbox{ for } i =1,\ldots,n-1 \right\} \subset V_2.
 \end{align}
 For all $\theta \in L^2(\Sigma)^{n-1}$ and all $\psi \in L^2(\Sigma,\spaceHn)$ it holds that (with a similar calcluation as above and using the weak equation for $h_i$)
 \begin{align*}
     (\phi_{\theta} , \psi)_{L^2(\Sigma,\spaceHn)} &= \sum_{i=1}^{n-1} \int_{\Sigma} \left(B^{-1}\theta(\x)\right)^i \int_{Y_f} \psi^i dy  d\x 
     \\
     &= \left(B^{-1} \theta , \int_{Y_f} (\psi^1,\ldots,\psi^{n-1})^T dy \right)_{L^2(\Sigma)}.
 \end{align*}
 This immediately implies $\eqref{CharacterizationRF}$. 
 
 Now, let $G \in V_2^{\perp} \subset L^2(\Sigma,\spaceHn)'$. We define the operator $\widetilde{G}: L^2(\Sigma)^{n-1} \rightarrow \R$ by the composition $\widetilde{G} := G \circ F$, more precisely we have 
 \begin{align*}
     \widetilde{G}(\theta) = \langle G, \phi_{\theta} \rangle_{L^2(\Sigma,\spaceHn)',L^2(\Sigma,\spaceHn)}.
 \end{align*}
 $\widetilde{G}$ is linear and bounded, and we can identify it with an element $g \in L^2(\Sigma)^{n-1}$. It holds the following well known (and easy to check) decomposition
 \begin{align}\label{DecompositionL2}
     L^2(\Sigma)^{n-1} = L_{\div}^2(\Sigma) \oplus \nabla_{\x} H_0^1(\Sigma),
 \end{align}
 with 
 \begin{align*}
     L_{\div}^2(\Sigma) := \left\{ \phi \in L^2(\Sigma)^{n-1} \, : \, \nabla_{\x} \cdot \phi = 0 \right\}.
 \end{align*}
 Since the mean of the first $n-1$ components of $\phi_{\theta} $ is equal to $\theta$, see $\eqref{MeanValueF}$, we have for all $\theta \in L_{\div}^2(\Sigma)$ that $\phi_{\theta} \in V_2$. This implies for all $\theta \in L_{\div}^2(\Sigma)$
 \begin{align*}
     (g, \theta )_{L^2(\Sigma)} = \langle G , \phi_{\theta}\rangle_{L^2(\Sigma,\spaceHn)',L^2(\Sigma,\spaceHn)} = 0.
 \end{align*}
 Now, the decomposition $\eqref{DecompositionL2}$ above implies the existence of $p \in H_0^1(\Sigma)$ with $g = \nabla_{\x} p$. Especially we obtain using again the mean value property $\eqref{MeanValueF}$ for every $\theta \in L^2(\Sigma)^{n-1}$
 \begin{align}\label{IdentityGwithGradpOnRF}
      \langle G , \phi_{\theta}\rangle_{L^2(\Sigma,\spaceHn)',L^2(\Sigma,\spaceHn)} = (g,\theta)_{L^2(\Sigma)} = (\nabla_{\x} p , \phi_{\theta})_{L^2(\Sigma \times Y_f)}.
 \end{align}
 In other words $G = \nabla_{\x} p$ on $R(F)$. Since $R(F) $ is closed we have $L^2(\Sigma,\spaceHn) = R(F) \oplus R(F)^{\perp}$, and because $R(F)^{\perp} \subset V_2$ the operator $G$ vanishes on $R(F)^{\perp}$.  Now, let $\phi \in L^2(\Sigma,\spaceHn)$ with decomposition $\phi = \phi_{\theta} + \phi^{\perp}$ for $\theta \in L^2(\Sigma)^{n-1}$ and $\phi^{\perp} \in R(F)^{\perp}$. Especially we have, see $\eqref{CharacterizationRF}$, that $\int_{Y_f} (\phi^{\perp} )^idy = 0$  for $i=1,\ldots,n-1$, and therefore
 \begin{align*}
     (\nabla_{\x} p , \phi^{\perp} )_{L^2(\Sigma \times Y_f)} =0.
 \end{align*}
 Altogether, we obtain with $\eqref{IdentityGwithGradpOnRF}$
 \begin{align*}
     \langle G, \phi \rangle_{L^2(\Sigma,\spaceHn)', L^2(\Sigma,\spaceHn)} &= \langle G, \phi_{\theta}\rangle_{L^2(\Sigma,\spaceHn)', L^2(\Sigma,\spaceHn)} 
     \\
     &= (\nabla_{\x} p, \phi_{\theta})_{L^2(\Sigma\times Y_f)} = (\nabla_{\x} p , \phi)_{L^2(\Sigma \times Y_f)},
 \end{align*}
 which gives the representation of $\spaceV^{\perp}$.
 \\
 
It remains to establish the uniqueness of the decomposition. Let $v^{\ast} \in \spaceV^{\perp}$ and $v^{\ast} = \nabla_{\x} p_i + \nabla_y q_i$ with $p_i \in H_0^1(\Sigma)$ and $q_i \in L^2(\Sigma \times Y_f)/\R$ for $i=1,2$. For every $\theta \in L^2(\Sigma)^{n-1}$ there exists $\phi \in L^2(\Sigma,\spaceHndiv)$ such that $\int_{Y_f} \phi^i dy = \theta^i$. We only sketch the proof, since the construction follows by similar arguments as for the operator $F$ above, where we additionally have to take into account that $\phi$ has to be solenoidal with respect to $y$. We define for almost every $(\x,y) \in \Sigma \times Y_f$
\begin{align*}
    \phi(\x,y):= \sum_{i=1}^{n-1} \left(\tilde{K}^{-1} \theta(\x)\right)^i w_i,
\end{align*}
where $w_i$ is the solution of the cell problem $\eqref{CellProblemWi}$  below and $\tilde{K}_{lk}:= K_{lk}$ for $l,k=1,\ldots,n-1$ with $K$ the permeability tensor defined in $\eqref{DefinitionEffectiveTensors}$. It is easy to check that this function has the desired properties. Hence, we obtain for every $\theta \in L^2(\Sigma)^{n-1}$ with $\nabla_{\x} \cdot \theta \in L^2(\Sigma)$ (using $ \phi \in V_1$)
\begin{align*}
    0 &= \langle \nabla_{\x} (p_1 - p_2) + \nabla_y (q_1 - q_2), \phi \rangle_{L^2(\Sigma,\spaceHn)',L^2(\Sigma,\spaceHn)}
    \\
     &= -(p_1 - p_2, \nabla_{\x} \cdot \phi )_{L^2(\Sigma \times Y_f)} = -(p_1 - p_2, \nabla_{\x} \cdot \theta )_{L^2(\Sigma)}.
\end{align*}
This implies $p_1 = p_2 =: p$. Hence, we obtain $v^{\ast} - \nabla_{\x} p = \nabla_y q_i \in V_1^{\perp}$ for $i=1,2$. It is easy to check that this representation is unique and therefore we have $q_1 = q_2$, what finishes the proof. 
 \end{proof}

\begin{rem}\label{RemarkOrthogonalComplement}\
\begin{enumerate}[label = (\roman*)]
    \item Lemma \ref{CharacterizationComplementW} can be easily modified for the problem with no-slip condition on $\Gamma_N^{\epsilon}$, i.e., for $u^{\epsilon} = 0$ on $\Gamma_N^{\epsilon}$. In this case, we have to replace $\spaceV$ by 
    \begin{align*}
        \widetilde{\spaceV} = \left\{ \phi \in \spaceV \, : \, \int_{Y_f} \phi dy \cdot \hat{n} = 0 \mbox{ on } \partial \Sigma\right\},
    \end{align*}
    and the space $V_2$ with 
    \begin{align*}
        \tilde{V}_2 = \left\{ \phi \in V_2 \, : \, \int_{Y_f} \phi dy \cdot \hat{n} = 0 \mbox{ on } \partial \Sigma\right\}.
    \end{align*}
    In this case, we obtain the same characterization for $\widetilde{\spaceV}^{\perp}$ as for $\spaceV^{\perp}$ with $p_0 \in H^1(\Sigma)$, where in the proof we only have to replace the decomposition $\eqref{DecompositionL2}$ with the usual Helmholtz decomposition
    \begin{align*}
        L^2(\Sigma)^{n-1} = L_{\sigma}^2(\Sigma) \oplus \nabla_{\x} H^1(\Sigma),
    \end{align*}
    with $L_{\sigma}^2(\Sigma) = \left\{ \phi \in L^2(\Sigma)^{n-1} \, : \, \nabla_{\x} \cdot \phi = 0 \mbox{ in } \Sigma, \, \phi \cdot \hat{n} = 0 \mbox{ on } \partial \Sigma\right\}$.
    
        \item\label{ItemRemarkCharacterizationSNpmNotEmpty} A similar result is  valid for $S_N^{\pm} \neq \emptyset $. In fact, it holds
    \begin{align*}
         \spaceV^\perp = \left\{\nabla_{\x} p_0 + \nabla_y p_1 \in L^2(\Sigma,\spaceHn)' \, : \,  p_0 \in H_0^1(\Sigma), \, p_1 \in L^2(\Sigma \times Y_f) \right\}.
    \end{align*}
    The only difference is, that the pressure $p_1$ in the decomposition is unique. This follows from the characterization of $V_1^{\perp}$, where now the divergence operator $\div_y$ is surjective from $\spaceHn$ to $L^2(Y_f)$. This leads (by the same arguments as for the case $S_N^{\pm}= \emptyset$)
    \begin{align*}
        V_1^{\perp} =  \left\{\nabla_y p_1 \, : \, p_1 \in L^2(\Sigma \times Y_f)\right\}.
    \end{align*}
    In the proof for the characterization of $V_2^{\perp}$ we only have the add in problem $\eqref{AuxiliaryCellProblemHi}$ the boundary condition 
    \begin{align*}
        -\nabla_y h_i\nu = 0 \qquad \mbox{ on } S^{\pm}_N.
    \end{align*}
    The crucial point is that the weak formulation $\eqref{AuxiliaryCellProblemHiWeak}$ is still the same.
    
    \item Our proof of Lemma \ref{CharacterizationComplementW} is also valid for a no-slip condition on the inner boundary $\Gamma_D^{\epsilon}$. In this case we just have to replace $\spaceHn$ with $H_{\#,0}^1(Y_f)^n$. In the case $S_N^{\pm} \neq \emptyset$ the pressure $p_1$ is unique as in \ref{ItemRemarkCharacterizationSNpmNotEmpty} above.
    
\end{enumerate}
\end{rem}

\section{Derivation of the macroscopic models}\label{Derivation:sec}

In this section we derive the macroscopic problems for different values for $\alpha$,  $\gamma$, and for the cases $S_N^{\pm}$ empty or not. This derivation is based on the two-scale compactness results in Lemma \ref{ConvergenceResults}. It is enough to establish a macroscopic equation for the limit fluid velocity $u_0$ on the space of test-functions $\spaceV$, see Section \ref{SectionTwoPressureDecomposition}. The associated pressures are then obtained by the orthogonality result in Lemma \ref{CharacterizationComplementW}. Since the steps are quite similar for the different cases, we will focus on the case $ \alpha = 0$ and $S_N^{\pm} = \emptyset$. For the other cases we shortly explain the  arising differences.
\\


\subsection{The case $\alpha = 0$ and $S_N^{\pm } = \emptyset$ (Proof of Theorem \ref{MainResultAlphaZeroSNempty})}
\label{SectionCaseAlphaZeroSNZero}

To pass to the limit in the microscopic variational equation  $\eqref{MicroModelNavierSlipVarEqu}$ we choose smooth test-functions. Therefore, we introduce the dense subset $\spaceV^{\infty}$ of  the space $\spaceV$ defined by
\begin{align*}
   \spaceV^{\infty} :=\left\{ \phi \in C^{\infty}(\overline{\Sigma} , \spaceHndiv) \, : \, \nabla_{\x} \cdot \int_{Y_f} \phi(\cdot_{\x},y) dy = 0 \text{ in } \Sigma\right\}.
\end{align*}
As a test-function in the microscopic equation $\eqref{MicroModelNavierSlipVarEqu}$ we choose $\phi^\eps(x):= \eps^{-1} \phi\left(\x,\dfrac{x}{\eps}\right)$ with $\phi \in C^{\infty}(\overline{\Sigma} , \spaceHndiv)  $  and obtain
\begin{align}
\begin{aligned}\label{MicroEquationTestedOscillatingFunctions}
    \frac{2\mu}{\eps} \int_{\Omega^\eps} \eps^{-1}& D(u^\eps) : \left[ \eps D_{\x}(\phi)\left(\x,\dfrac{x}{\eps}\right) + D_y (\phi)\left(\x,\dfrac{x}{\eps}\right)\right]  dx 
    \\
    -& \frac{1}{\eps} \int_{\Omega^{\eps}} (p^\eps - p_b^\eps) \nabla_{\x} \cdot \phi \left(\x,\dfrac{x}{\eps}\right) dx 
    \\
    &=  \frac{1}{\eps}\int_{\Omega^{\eps}} (f^\eps - \nabla p_b^\eps) \cdot \phi\left(\x,\dfrac{x}{\eps}\right) dx + \frac{1}{\eps}\int_{\Gamma_D^\eps } g^\eps \cdot \phi\left(\x,\dfrac{x}{\eps}\right) d\sigma 
\end{aligned}
   \end{align}
Using the convergence results from Lemma \ref{ConvergenceResults} and the assumptions on the data  we obtain for $\eps \to 0$
\begin{align}\label{AuxiliaryEquationMacroModel1}
\begin{aligned}
       2\mu \int_{\Sigma}\int_{Y_f}& D_y(u_0) : D_y(\phi) dy d\x -\int_{\Sigma} \int_{Y_f } (p_0 - p_b) \nabla_{\x} \cdot \phi dy d\x
    \\
    &= \int_{\Sigma } \int_{Y_f } (f^0 - \nabla_{\x} p_b - \nabla_y p_b^1) \cdot \phi dy d\x + \int_{\Sigma} \int_{\Gamma_D } g^0 \cdot \phi d\sigma_y d\x
    \\
    &= \int_{\Sigma } \int_{Y_f } (f^0 - \nabla_{\x} p_b) \cdot \phi dy d\x + \int_{\Sigma} \int_{\Gamma_D } g^0 \cdot \phi d\sigma_y d\x
\end{aligned}
\end{align}
We emphasize that the last equation is only valid for $S^{\pm}_N = \emptyset$.
By density this equation is also valid for all functions in $L^2(\Sigma,\spaceHndiv)$.
For  $\phi \in \spaceV^{\infty}$ the second term on the left-hand side vanishes and we obtain
\begin{align*}
    2\mu \int_{\Sigma}\int_{Y_f}& D_y(u_0) : D_y(\phi) dy d\x 
    \\
    &= \int_{\Sigma } \int_{Y_f } (f^0 - \nabla_{\x} p_b) \cdot \phi dy d\x   + \int_{\Sigma} \int_{\Gamma_D } g^0 \cdot \phi d\sigma_y d\x.
\end{align*}
By density this result holds for all $\phi \in \spaceV$.
Due to Lemma \ref{CharacterizationComplementW}, there exist $P_0 \in H_0^1(\Sigma)$ and $p_1 \in L^2(\Sigma\times Y_f)/\R$ such that
\begin{align*}
        2\mu \int_{\Sigma}\int_{Y_f}& D_y(u_0) : D_y(\phi) dy d\x + \int_{\Sigma} \int_{Y_f } \nabla_{\x} P_0 \cdot \phi dy d\x - \int_{\Sigma } \int_{Y_f } p_1 \nabla_y \cdot \phi dy d\x
    \\
    &= \int_{\Sigma } \int_{Y_f } (f^0 - \nabla_{\x} p_b) \cdot \phi dy d\x  + \int_{\Sigma} \int_{\Gamma_D } g^0 \cdot \phi d\sigma_y d\x.
\end{align*}
for all $\phi \in L^2(\Sigma, \spaceHn)$. From the uniqueness of the decomposition in Lemma \ref{CharacterizationComplementW} we immediately obtain that $P_0 = p_0 - p_b$.
Altogether we showed that the triple $(u_0, p_0, p_1) $ solves the variational equation  $\eqref{VariationalEquationTwoPressureProblem}$ and therefore is a weak solution of the macroscopic model $\eqref{TwoScaleMacroModelAlpha0}$. Uniqueness of the solution follows by standard arguments and we skip the proof.

To finish the proof of Theorem \ref{MainResultAlphaZeroSNempty} we have to derive the Darcy-law for $p_0$. Due to the linearity of the problem $\eqref{TwoScaleMacroModelAlpha0}$, we easily obtain the following representation for $u_0$ and $p_1$
\begin{align}
\begin{aligned}\label{RepresentationVelocityPressure}
    u_0(\x,y) &= \frac{1}{\mu} \sum_{i=1}^{n-1} \left(f^0_i(\x) - \partial_{x_i} p_0(\x)\right) w_i (y) + \frac{1}{\mu}g_{\Sigma}(\x) w_{\Gamma_D}(y),
    \\
    p_1(\x,y) &= \sum_{i=1}^{n} \left(f^0_i(\x)- \partial_{x_i}p_0(\x) \right) \pi_i(y) + g_{\Sigma}(\x) \pi_{\Gamma_D}(y), 
\end{aligned}
\end{align}
where $(w_i,\pi_i) \in \spaceHndiv \times L^2(Y_f)/\R$ for $i=1,\ldots,n$, and $(w_{\Gamma_D},\pi_{\Gamma_D} ) \in \spaceHndiv \times L^2(Y_f)/\R$ are the unique weak solutions of the following cell problems:
\begin{align}
\begin{aligned}\label{CellProblemWi}
-2 \nabla_y \cdot (D_y(w_i)) + \nabla_y \pi_i &= e_i &\mbox{ in }& Y_f,
\\
\nabla_y \cdot w_i &= 0 &\mbox{ in }& Y_f,
\\
w_i \cdot \hat{n} &= 0 &\mbox{ on }& \Gamma_D,
\\
2[D_y(w_i)\hat{n}]_{\tau} &= 0 &\mbox{ on }& \Gamma_D,
\\
w_i &\mbox{ is } (0,1)^{n-1}\mbox{-periodic},
\end{aligned}
\end{align}
and 
\begin{align}
\begin{aligned}\label{CellProblemWGamma}
-2 \nabla_y \cdot (D_y(w_{\Gamma_D})) + \nabla_y \pi_{\Gamma_D} &= 0 &\mbox{ in }& Y_f,
\\
\nabla_y \cdot w_{\Gamma_D} &= 0 &\mbox{ in }& Y_f,
\\
w_{\Gamma_D} \cdot \hat{n} &= 0 &\mbox{ on }& \Gamma_D,
\\
2[D_y(w_{\Gamma_D})\hat{n}]_{\tau} &= g_{\Gamma_D} &\mbox{ on }& \Gamma_D,
\\
w_{\Gamma_D} &\mbox{ is } (0,1)^{n-1}\mbox{-periodic}.
\end{aligned}
\end{align}
Obviously, we have $(w_n,\pi_n) = (0,y_n)$, and therefore it is enough to consider in the representation for $u_0$ in $\eqref{RepresentationVelocityPressure}$ the sum from $1$ to $n-1$. 

Further, we define the tensor $K \in \R^{n\times n}$  and the effective vector $\beta \in \R^n$ by ($i,j=1,\ldots,n$)
\begin{align}
\begin{aligned}\label{DefinitionEffectiveTensors}
    K_{ij} &:= \int_{Y_f} D_y(w_i): D_y (w_j) dy,
    \\
    \beta &:= \int_{Y_f} w_{\Gamma_D} dy.
\end{aligned}
\end{align}
\begin{rem}
Of course, since $w_n = 0$, it holds for $i=1,\ldots,n$ that 
\begin{align*}
    K_{ni} = K_{in} = 0.
\end{align*}
However, to keep the notation a little bit simpler, we consider $K$ as a $(n\times n)$-matrix. Further, a simple calculation shows (using again $S_N^{\pm} = \emptyset$)
\begin{align*}
    \beta^n = 0.
\end{align*}
\end{rem}

Now, we define the Darcy-velocity $\bar{u}_0: \Sigma \rightarrow \R^n$ by
\begin{align*}
    \bar{u}_0(\x) := \int_{Y_f} u_0(\x,y) dy = \frac{1}{\mu} K(f^0(\x) - \nabla_{\x} p_0(\x)) +\frac{1}{\mu} \beta g_{\Sigma}(\x).
\end{align*}
Using $\nabla_{\x} \cdot \bar{u}_0 = 0$, we obtain the Darcy law $\eqref{DarcyLawAlphaZero}.$

\begin{rem}
We have $\bar{u}_0^n = 0$ and $\bar{u}_0 $ is independent of $(f^0)^n$. This force only has an influence on the pressure $p_1.$
\end{rem}

\subsection{The case $\alpha = 0$ and  $S_N^{\pm} \neq \emptyset$ (Proof of Theorem \ref{MainResultAlphaZeroSNnotempty})}
\label{SectionAlphaZeroSNneqempty}
 For  $S_N^{\pm} \neq \emptyset$ we need an additional boundary condition in problem $\eqref{TwoScaleMacroModelAlpha0}$ on $S_N^{\pm}$. It is easy to check that the first equality in $\eqref{AuxiliaryEquationMacroModel1}$ is still valid. However, now the term including $p_b^1$ does not vanish. More precisely, we have for all $\phi \in L^2(\Sigma, \spaceHndiv)$ 
\begin{align*}
    2\mu \int_{\Sigma}&\int_{Y_f} D_y(u_0) : D_y(\phi) dy d\x -\int_{\Sigma} \int_{Y_f } (p_0 - p_b) \nabla_{\x} \cdot \phi dy d\x
    \\
    &= \int_{\Sigma } \int_{Y_f } (f^0 - \nabla_{\x} p_b ) \cdot \phi dy d\x \int_{\Sigma} + \sum_{\pm}\int_{\Sigma}\int_{S_N^{\pm}}p_b^1 \phi \cdot \hat{n} d\sigma_y d\x +   \int_{\Gamma_D } g^0 \cdot \phi d\sigma_y d\x.
\end{align*}
Using similar arguments as above, see also Remark \ref{RemarkOrthogonalComplement}, we obtain $p_1 \in L^2(\Sigma \times Y_f)$ such that the triple $(u_0,p_0,p_1)$ solves problem $\eqref{TwoScaleMacroModelAlpha0}$ with the additional boundary condition
\begin{align*}
    \left[ p_1 I - 2\mu D_y(u_0)\right]\hat{n} = p_b^1 \hat{n} \qquad\mbox{on } S_N^{\pm}.
\end{align*}
Using the separation for $p_b^1$ from assumption \ref{AssumptionPressure}, we obtain
\begin{align*}
        u_0(\x,y) &= \frac{1}{\mu} \sum_{i=1}^{n} \left(f^0_i(\x) - \partial_{x_i} p_0(\x)\right) w_i (y) + \frac{1}{\mu}g_{\Sigma}(\x) w_{\Gamma_D}(y) + \frac{1}{\mu} p_{b,\Sigma}^1(\x) w_{N}(y),
    \\
    p_1(\x,y) &= \sum_{i=1}^{n} \left(f^0_i(\x)- \partial_{x_i}p_0(\x) \right) \pi_i(y) + g_{\Sigma}(\x) \pi_{\Gamma_D}(y) + p_{b,\Sigma}^1(\x) \pi_{N}(y), 
\end{align*}
where $(w_i,\pi_i) \in \spaceHndiv \times L^2(Y_f)$ for $i=1,\ldots,n$, and $(w_{\Gamma_D},\pi_{\Gamma_D} ) \in \spaceHndiv \times L^2(Y_f)$ solve the cell problems $\eqref{CellProblemWi}$ and $\eqref{CellProblemWGamma}$ with the additional boundary conditions
\begin{align*}
    [\pi_i I - 2D_y(w_i)]\hat{n} &= 0 &\mbox{ on }& S_N^{\pm} ,
    \\
    [\pi_{\Gamma_D} I - 2D_y(w_{\Gamma_D})]\hat{n} &= 0 &\mbox{ on }& S_N^{\pm},
\end{align*}
and $(w_N,\pi_N) \in  \spaceHndiv \times L^2(Y_f)$ is the unique weak solution of the cell problem
\begin{align}
    \begin{aligned}\label{CellProblemWN}
    -2 \nabla_y \cdot (D_y(w_N)) + \nabla_y \pi_N &= 0 &\mbox{ in }& Y_f,
\\
\nabla_y \cdot w_N &= 0 &\mbox{ in }& Y_f,
\\
w_N \cdot \hat{n} &= 0 &\mbox{ on }& \Gamma_D,
\\
2[D_y(w_N)\hat{n}]_{\tau} &= 0 &\mbox{ on }& \Gamma_D,
\\
[\pi_N I - 2D_y(w_N)]\hat{n} &= p_{b,N} \hat{n} &\mbox{ on }& S_N^{\pm},
\\
w_N &\mbox{ is } (0,1)^{n-1}\mbox{-periodic}.
\end{aligned}
\end{align}
We emphasize that here the pressures $\pi_i$, $\pi_{\Gamma_D}$, and $\pi_N$ are unique (not only up to a constant) and $w_n$ is in general not equal to zero. Defining $K$ and $\beta$ as in $\eqref{DefinitionEffectiveTensors}$ (for the modified cell problems), and 
\begin{align*}
    \kappa := \int_{Y_f} w_N dy,
\end{align*}
we obtain
\begin{align*}
     \bar{u}_0(\x) = \frac{1}{\mu} K(f(\x) - \nabla_{\x} p_0(\x)) +\frac{1}{\mu} \beta g_{\Sigma}(\x) +  \frac{1}{\mu} \kappa p_{b,\Sigma}^1(\x),
\end{align*}
and $p_0$ solves the Darcy-law $\eqref{DarcyLawAlphaZeroSNneqempty}$. This finishes the proof of Theorem \ref{MainResultAlphaZeroSNnotempty}.
We emphasize that in this case $K_{ni} = K_{in}$ and also $\beta_n$ are not zero, and $K$ is even a positive matrix. Hence, $\bar{u}_0^n \neq 0$ and we have a Darcy-flow vertical to the thin layer.

\subsection{The case $\alpha > 0$ (Proof of Theorem \ref{MainResultAlphaNotZero})}

In the following we only consider the case $S_N^{\pm}  = \emptyset$, since the other cases can be treated in a similar way with slight modifications as in Section \ref{SectionAlphaZeroSNneqempty}.
For $\alpha \neq 0$ the limit equation is also dependent on the parameter $\gamma$. While the parameter $\gamma$ has no influence on the \textit{a priori} estimates of the fluid velocity and pressure in the layer $\Omega^{\epsilon}$, the fluid velocity at the oscillating surface $\Gamma_D^{\epsilon}$ is depending on $\gamma$. In fact, due to Lemma \ref{AprioriEstimatesVelocity}, it holds that
\begin{align*}
    \Vert \ueps \Vert_{L^2(\Gamma_D^{\epsilon})} \le C 
    \begin{cases}
        \epsilon^2 &\mbox{ for } \gamma \geq -1,
        \\
        \epsilon^{\frac{3-\gamma}{2}} &\mbox{ for } \gamma \leq -1.
    \end{cases}
\end{align*}
Hence, we can expect that the behavior of the limit model changes in the critical value $\gamma = -1$. Compared to the result for $\alpha =0$, for the homogenization we only have to deal in the microscopic equation $\eqref{MicroModelNavierSlipVarEqu}$ with the additional term
\begin{align*}
 \alpha \epsilon^{\gamma} \int_{\Gamma_D^{\epsilon}} \ueps \cdot \phi^{\epsilon} d\sigma.
\end{align*}
For the derivation of the macroscopic model we tested equation $\eqref{MicroModelNavierSlipVarEqu}$ with $\epsilon^{-1} \phi\left(\x,\frac{x}{\epsilon}\right)$ for $\phi \in C^{\infty}\left(\overline{\Sigma} , \spaceHndiv\right)$. Hence, to pass to the limit $\epsilon \to 0$ we have to identify the limit of the term
\begin{align*}
    B_{\epsilon}^{\gamma}:= \alpha \epsilon^{\gamma - 1} \int_{\Gamma_D^{\epsilon}} \ueps \cdot \phi \left(\x ,\frac{x}{\epsilon}\right)d\sigma.
\end{align*}
Using the trace inequality $\eqref{TraceInequality}$ and the Poincar\'e inequality $\eqref{PoincareInequalityHn}$, it is easy to check that
\begin{align*}
    \left\Vert \phi\left(\x,\frac{x}{\epsilon}\right) \right\Vert_{L^2(\Gamma_D^{\epsilon})} \le C \left(\sup_{\overline{x}\in \Sigma} \int_{Y_f} \lvert\nabla_y\phi(\overline{x},y)\rvert^2+ \epsilon^2\lvert\nabla_{\overline{x}}\phi(\overline{x},y)\rvert^2\,dy \right)^{1/2}.
\end{align*}
Together with the \textit{a priori} estimates from Lemma \ref{AprioriEstimatesVelocity}, we obtain
\begin{align}\label{EstimatesBepsGamma}
    \vert B_{\epsilon}^{\gamma} \vert \le C 
    \begin{cases}
    \epsilon^{\gamma + 1} &\mbox{ for } \gamma > -1,
    \\
    1 &\mbox{ for } \gamma = -1,
    \\
    \epsilon^{\frac{\gamma +1}{2}} &\mbox{ for } \gamma < -1.
    \end{cases}
\end{align}
For $\gamma < -1$ we cannot control in this way the term $B_{\epsilon}^{\gamma}$.  However, in this case we  will test with functions vanishing on the boundary $\Gamma_D$. This is consistent with the fact that the two-scale limit of $\ueps$ has zero trace on $\Gamma_D$ for $\gamma < -1$, see Lemma \ref{ConvergenceResults}, and the slip condition becomes a no-slip condition in the limit.
\begin{proof}[Proof of Theorem \ref{MainResultAlphaNotZero}]\ref{MainResultAlphaNotZeroGammaBigger-1}
 Let $\gamma > -1$. The proof follows the same lines as in the case $\alpha = 0$. The only difference occurs in equation $\eqref{MicroEquationTestedOscillatingFunctions}$, where the additional term $B_{\epsilon}^{\gamma}$ occurs on the left-hand side. However, due to estimate $\eqref{EstimatesBepsGamma}$, we have
 \begin{align*}
     B_{\epsilon}^{\gamma } \rightarrow 0 \qquad \mbox{ for } \epsilon \to 0.
 \end{align*}
 Hence, we obtain again equation $\eqref{AuxiliaryEquationMacroModel1}$ and the result follows by the same calculations as before.
\end{proof}

\begin{proof}[Proof of Theorem \ref{MainResultAlphaNotZero}\ref{MainResultAlphaNotZeroGammaEqual-1}]
Let  $\gamma = -1$. In this case we obtain a contribution from the term $B_{\epsilon}^{-1}$. In fact, from the two-scale convergence of the traces in Lemma \ref{ConvergenceResults} we obtain
\begin{align*}
    B_{\epsilon}^{-1} = \alpha \frac{1}{\epsilon^2} \int_{\Gamma_D^{\epsilon}} \ueps \cdot \phi\left(\x,\dfrac{x}{\epsilon}\right) d\sigma \overset{\epsilon \to 0}{\longrightarrow} \alpha \int_{\Sigma} \int_{\Gamma_D} u_0 \cdot \phi(\x,y) d\sigma_y d\x. 
\end{align*}
Hence, the boundary condition for the tangential normal stress on $\Gamma_D$ in the two-scale homogenized problem $\eqref{TwoScaleMacroModelAlpha0}$ has to be replaced by 
\begin{align*}
    2\mu(D_y(u_0)\hat{n})_{\tau} + \alpha u_0 &= g_0 &\mbox{ on }& \Sigma \times \Gamma_D,
\end{align*}
and in the same way we have to modify the cell problems $\eqref{CellProblemWi}$ and $\eqref{CellProblemWGamma}$ (and $\eqref{CellProblemWN}$ for $S_N^{\pm} \neq \emptyset$).
\end{proof}

\begin{proof}[Proof of Theorem \ref{MainResultAlphaNotZero}\ref{MainResultAlphaNotZeroGammaSmaller-1}] Let $\gamma <-1$. From Lemma \ref{ConvergenceResults} we obtain $u_0 = 0$ on $\Gamma_D$. Now, we test equation $\eqref{MicroModelNavierSlipVarEqu}$ with $\epsilon^{-1} \phi\left(\x,\frac{x}{\epsilon}\right)$ for $\phi \in C^{\infty}\left(\overline{\Sigma}, H_{\#,0}^1(Y_f)^n\right) $ such that $\nabla_y \cdot \phi = 0$. Hence, the boundary terms on $\Gamma_D^{\epsilon}$ vanish and we obtain for $\epsilon \to 0$ the identity $\eqref{AuxiliaryEquationMacroModel1}$ without the boundary term on $\Gamma_D$. By density this equation is valid for all functions  $\phi \in L^2(\Sigma, H_{\#,0}^1(Y_f)^n)$ with $\nabla_y \cdot \phi=0$. Using again Lemma \ref{CharacterizationComplementW} and Remark \ref{RemarkOrthogonalComplement}, we obtain the desired result. 
 
\end{proof}

\section{Conclusion}

We derived Darcy-laws for Stokes flow through a thin perforated layer with a Navier-slip condition on the oscillating boundary of the perforations. For this we established multi-scale techniques combining dimension reduction and homogenization, such as the restriction operator for thin domains, a two pressure decomposition for functionals, and an $\epsilon$-uniform Korn inequality for vector fields with vanishing normal traces. We emphasize that the Korn inequality is not restricted to thin perforated layers, but can also be used for the treatment of perforated domains, which seems to be open in the literature on homogenization. 

In our paper we only considered the case that the thickness of the layer and the periodicity within the layer (as well as the size of the solid phase $Y_s$) are of order $\epsilon$. However, other scalings for the geometry are also important for applications, and we expect that our methods and results give a contribution to solve such problems.

In many applications the thin layer is coupled to other regions, like free flow domains, other thin layers leading to multi-layered systems, or porous medium with elastic properties. In this case, coupling conditions between the different regions are needed and other external boundary conditions are necessary. The multi-scale techniques developed in this paper might be an important step in the treatment of such more complex geometrical structures including thin layers. However, such problems lead to additional difficulties like the choice of test-functions, and often one has to take into account other concepts like boundary layers.

\appendix

\section{Two-scale convergence in thin domains}
\label{SectionTSConvergence}

We briefly introduce two-scale convergence concepts for thin layers \cite{BhattacharyaGahnNeussRadu,NeussJaeger_EffectiveTransmission}, and recall the compactness results used in this paper.
\begin{dfn}\
\begin{enumerate}
[label = (\roman*)]
\item\,[Two-scale convergence in the thin layer $\oem$] We say the sequence $\weps \in L^2( \oem)$ converges (weakly) in the two-scale sense to a limit function $w_0 \in L^2(\Sigma \times Y)$ if 
\begin{align*}
\lim_{\epsilon\to 0} \foe \int_{\oem} \weps(x) \phi \bxfxe dx = \int_{\Sigma} \int_Y w_0(\x,y) \psi(\x,y) dy d\x 
\end{align*}
for all $\phi \in L^2(\Sigma,C_{\#}^0(\overline{Y}))$. We write 
\begin{align*}
\weps \rats w_0.
\end{align*}
\item\,[Two-scale convergence on the oscillating surface $\Gamma_D^{\epsilon}$] We say the sequence $\weps \in L^2( \Gamma_D^{\epsilon})$ converges (weakly) in the two-scale sense to a limit function $w_0 \in L^2( \Sigma \times \Gamma_D)$ if 
\begin{align*}
\lim_{\epsilon\to 0}   \int_{\Gamma_D^{\epsilon}} \weps(x) \phi \bxfxe d\sigma  =  \int_{\Sigma} \int_{\Gamma_D} w_0(\x,y) \psi(\x,y) d\sigma_y d\x 
\end{align*}
for all $\phi \in C^0(\overline{\Sigma},C_{\#}^0(\Gamma_D))$.  We write  
\begin{align*}
\weps \rats w_0 \qquad \mbox{on } \Gamma_D^{\epsilon}.
\end{align*}
\end{enumerate}
\end{dfn}
The following lemma gives basic compactness results for the two-scale convergence in thin layers. 
\begin{lem}\label{LemmaBasicTSCompactness}\
\begin{enumerate}
[label = (\roman*)]
\item Let $\weps \in H^1(\oem)$ be a sequence with
\begin{align*}
\frac{1}{\sqrt{\epsilon}}\Vert \weps \Vert_{L^2( \oem)} + \sqrt{\epsilon}\Vert \nabla \weps \Vert_{L^2( \oem)}  \le C.
\end{align*}
Then there exists a subsequence (again denoted $\weps$) and a limit function $w_0 \in  L^2(  \Sigma, H_{\#}^1(Y))$ such that the following two-scale convergences hold
\begin{align*}
\weps  &\rats w_0,
\\
\epsilon \nabla \weps &\rats \nabla_y w_0.
\end{align*}
\item  Consider the sequence $\weps \in L^2( \Gamma_D^{\epsilon})$ with
\begin{align*}
\Vert \weps \Vert_{L^2(\Gamma_D^{\epsilon})} \le C.
\end{align*}
Then there exists a subsequence (again denoted $\weps$) and a limit function $w_0 \in  L^2(  \Sigma\times \Gamma_D)$ such that  
\begin{align*}
\weps \rats w_0 \qquad \mbox{on } \Gamma_D^{\epsilon}.
\end{align*}
\end{enumerate}
\end{lem}
\begin{proof}
 For the first result we refer to \cite{NeussJaeger_EffectiveTransmission}. A proof for the second result can be found in \cite{BhattacharyaGahnNeussRadu} for surfaces of thin channels. However, the proof is valid with the same arguments for our geometrical setting.
\end{proof}
Since we are working on perforated domains $\Omega^{\epsilon}$, for an application of the two-scale convergence we have to extend the functions to the whole layer $\Omega_{\epsilon}^M$. The easiest way is to use the zero extension, but we loose the spatial regularity. Using the results from \cite{Acerbi1992}, we obtain for every $\weps \in H^1(\Omega^{\epsilon})$ the existence of an extension $\tweps\in H^1(\Omega_{\epsilon}^M)$, such that
\begin{align*}
    \Vert \tweps \Vert_{L^2(\Omega_{\epsilon}^M)} \le C \Vert \weps \Vert_{L^2(\Omega^{\epsilon})}, \qquad \Vert \nabla \tweps \Vert_{L^2(\Omega_{\epsilon}^M)} \le C \Vert \nabla \tweps \Vert_{L^2(\Omega^{\epsilon})}.
\end{align*}
\begin{lem}\label{TwoScaleCompactnessPerforated}
Let $\weps \in H^1(\Omega^{\epsilon})$ be a sequence with 
\begin{align*}
    \frac{1}{\sqrt{\epsilon}}\Vert \weps \Vert_{L^2( \Omega^{\epsilon})} + \sqrt{\epsilon}\Vert \nabla \weps \Vert_{L^2( \Omega^{\epsilon})}  \le C.
\end{align*}
Then, there exists $w_0 \in L^2(\Sigma, H_{\#}^1(Y_f))$, such that up to a subsequence it holds that 
\begin{align*}
   \chi_{\Omega^{\epsilon}} \weps  &\rats \chi_{Y_f} w_0,
\\
\chi_{\Omega^{\epsilon}}\epsilon \nabla \weps &\rats \chi_{Y_f} \nabla_y w_0,
\\
\weps\vert_{\Gamma_D^{\epsilon}} &\rats w_0\vert_{\Gamma_D}.
\end{align*}
\end{lem}
\begin{proof}
 Such kind of results are well known in the two-scale convergence theory. However, for the sake of completeness we give a sketch for the proof. From the scaled trace inequality $\eqref{TraceInequality}$ and the bounds for the extension $\tweps$, we obtain
 \begin{align*}
     \Vert \tweps \Vert_{L^2(\Gamma_D^{\epsilon})} + \frac{1}{\sqrt{\epsilon}} \Vert \tweps \Vert_{L^2(\Omega_{\epsilon}^M)} + \sqrt{\epsilon} \Vert \nabla \tweps \Vert_{L^2(\Omega_{\epsilon}^M)} \le C.
 \end{align*}
 Hence, the convergence results for $\chi_{\Omega^{\epsilon}} \weps$ and $\chi_{\Omega^{\epsilon}} \epsilon \nabla \weps $ follow directly from Lemma \ref{LemmaBasicTSCompactness}. Further, we obtain that the trace of $\weps$ also converges (up to a subsequence) in the two-scale sense, and the limit function can easily be identified with the trace of $w_0$ by integration by parts.
\end{proof}

\section{The Bogovski\u{i} operator}
\label{SectionAppendixBogovskii}

In the following we denote by $\Omega \subset \R^n$ an arbitrary bounded Lipschitz domain and  $\Gamma \subset \partial \Omega$.
A famous result of Bogovski\u{i} \cite{BogovskiSolutionFirstBVP}, see also \cite{BorchersSohr}, says that there exists a linear and bounded operator $B:L^2_0(\Omega) \rightarrow H_0^1(\Omega)^n$ ($L^2_0(\Omega)$ is the space of $L^2$-functions with mean value zero) such that
\begin{align*}
    \nabla \cdot Bf = f. \qquad \mbox{for all } f \in L^2_0(\Omega).
\end{align*}
This implies the following result (see $\eqref{SpaceZeroTracesPartBoundary}$ for the definition of the space $H^1(\Omega,\Gamma)$):

\begin{lem}\label{LemmaBogovskii}
Let $\vert \partial \Omega  - \Gamma \vert > 0 $. Then there exists a linear and bounded operator $B_{\Gamma}: L^2(\Omega) \rightarrow H^1(\Omega,\Gamma)^n$, such that
\begin{align*}
    \nabla \cdot B_{\Gamma} f = f \qquad \mbox{for all } f \in L^2(\Omega).
\end{align*}
\end{lem}
\begin{proof}
 A proof can be found in \cite[Theorem 5.4]{fabricius2017stokes}.
\end{proof}

%

\bibliographystyle{abbrv}
\bibliography{Literature_ThinDomainsPressureBC}
\end{document}